\documentclass[12pt]{amsart}
\usepackage{amssymb}
\usepackage{enumerate,color}
\usepackage{graphicx}
\usepackage[text={6in,9in},centering]{geometry}
\usepackage{subfigure}

\theoremstyle{plain}
\newtheorem{thm}{Theorem}[section]
\newtheorem{cor}[thm]{Corollary}
\newtheorem{lem}[thm]{Lemma}
\newtheorem{prop}[thm]{Proposition}
\theoremstyle{definition}
\newtheorem{defn}{Definition}
\newtheorem{exmp}{Example}
\theoremstyle{remark}
\newtheorem{rem}{Remark}

\def\R{\mathbb{R}}

\newcommand{\bu}{\mathbf{u}}
\newcommand{\bv}{\mathbf{v}}

\newcommand{\w}{\omega} 
\newcommand{\be}{\begin{equation}} 
	\newcommand{\ee}{\end{equation}} 

\newcommand\pkd{\mbox{\rm dim}_{\rm P}\,} 
\newcommand\hdd{\mbox{\rm dim}_{\rm H}\,} 
\newcommand\ubd{\overline{\mbox{\rm dim}}_{\rm B}\,} 
\newcommand\uid{\overline{\mbox{\rm dim}}_{\rm \theta}\,} 
\newcommand\lid{\underline{\mbox{\rm dim}}_{\rm \theta}\,} 

\newcommand{\red}[1]{{\color{red}#1}}

\title[Non-autonomous iterated function systems]{Dimensions and dimension spectra of Non-autonomous iterated function systems}

\date{}

\author[J. J. Miao]{Jun Jie Miao}
\address[J. J. Miao]{School of Mathematical Sciences,  Key Laboratory of MEA(Ministry of Education) \& Shanghai Key Laboratory of PMMP,  East China Normal University, Shanghai 200241, P.~R. China}
\email{jjmiao@math.ecnu.edu.cn}

\author[Tianrui Wang]{Tianrui Wang}
\address[Tianrui Wang]{School of Mathematical Sciences,  Key Laboratory of MEA(Ministry of Education) \& Shanghai Key Laboratory of PMMP,  East China Normal University, Shanghai 200241, P.~R. China}
\email{51265500036@stu.ecnu.edu.cn}

\begin{document}

\begin{abstract}
Non-autonomous iterated function systems are a generalization of iterated function systems. If the contractions in the system are conformal mappings, it is called a non-autonomous conformal iterated function system, and its attractor is called a non-autonomous conformal set.
In this paper, we study intermediate  dimension spectra of non-autonomous conformal sets which  provide a unifying framework for Hausdorff
and box-counting dimensions. First, we obtain the intermediate dimension spectra formula of non-autonomous conformal sets by using   upper and lower topological pressures. As a consequence,     we obtain  simplified forms of their Hausdorff, packing  and box dimensions. Finally,  we explore the Hausdorff dimensions of the non-autonomous infinite conformal iterated function systems which consists of  countably many conformal mappings at each level, and we provide the Hausdorff dimension formula under certain conditions
	\end{abstract}
	\maketitle

\section{Introduction}
\subsection{Dimensions of fractals}
Dimensions theory is central to fractal geometry, and  Hausdorff and box-counting dimensions are two fundamental notions used in fractal geometry and dynamical systems.  It is well-known that  Hausdorff and box-counting dimensions are identical for self-similar sets satisfying the open set condition, but there are also many interesting fractal sets with the two dimensions different. For example, the Hausdorff dimensions of many non-typical self-affine sets and Moran sets are strictly less than their (upper) box-counting dimensions; see  \cite{Baran07,Bedfo84, LalGa92,McMul84, Wen00} for details. The reason is because covering sets of widely ranging scales are permitted in the definition of Hausdorff dimensions, whereas  covering sets with the same size are essentially used in box-counting dimensions. We refer  readers to ~\cite{Fal,  Wen00} for background reading on fractal geometry. .

Recently, there is a large amount of  literature on dimension spectra which  provides a unifying framework for many   dimensions of fractal geometry; see \cite{Ban1, FFK,Fra1,FY} for various studies on dimension spectra. In \cite{FFK}, Falconer, Fraser and Kempton  introduced intermediate dimensions to provide a unifying framework for Hausdorff and box-counting dimensions. 	
	\begin{defn}
		Given a subset $E\subset \mathbb{R}^d$. For each $0\le \theta \le 1$, the lower and upper $\theta$-intermediate dimensions of $E$ are defined respectively by
		\begin{eqnarray*}
			\lid E&=&\inf\{  s\ge 0:\textit{for all }\varepsilon>0, \Delta>0, \textit{there exists } 0<\delta\le\Delta \textit{ and a cover $\{U_i\}$} \\
			&&\hspace{2cm} \textit{of $E$ such that } \delta^\frac{1}{\theta}\le \lvert U_i\rvert \le \delta, \sum \lvert U_i\rvert ^s\le \varepsilon\}, \\
			\uid E&=&\inf\{  s\ge 0:\textit{for all }\varepsilon>0,  \textit{ there exists } \Delta >0, \textit{such that for all }0<\delta\le\Delta \\
			&&\hspace{1cm}    \textit{ there is a cover $\{U_i\}$} \textit{of $E$ such that } \delta^\frac{1}{\theta}\le \lvert U_i\rvert \le \delta, \sum \lvert U_i\rvert ^s\le \varepsilon\}.
		\end{eqnarray*}
		If $\lid E=\uid E $, we call it the $\theta$-intermediate dimension of $E$, denoted by   $\dim_{\theta}E$.
	\end{defn}
Moreover, in \cite{FFK,Fra1}, the authors proved the continuity of the intermediate dimensions.
	\begin{prop}
		Given a bounded set  $E \subset \R^d$, the dimension spectra $\lid E$ and $\uid E$ are  continuous functions for  $\theta\in (0,1]$.
	\end{prop}
As we see from the definition, for $\theta = 0$, there are no restrictions for the size of the sets used in the covers which gives Hausdorff dimension. On the other hand, for $\theta = 1$, the only covers using sets of the same size are allowable which recovers box-counting dimension. Therefore, Hausdorff and box-counting dimensions are the special values of   intermediate dimensions $\dim_{\theta} E$, that is, 	
	$$
	\overline{\dim}_0 E =\underline{\dim}_0E=\hdd E, \quad  \overline{\dim}_1E=\overline{\dim}_B E, \quad   \underline{\dim}_1E=\underline{\dim}_B E,
	$$
It is  valuable to study intermediate dimension spectra for which $\theta$ the transition  occurs in geometric behavior, and this may deepen our understanding of  dimensions and the geometric structure of fractal sets; see \cite{BJ1, BanKol, BFF,Fra21, Kol1} for various related studies and applications.

Since intermediate dimension spectra provide a continuum between Hausdorff and box-counting dimensions, it is interesting to explore the dimensions spectra for the fractals with different Hausdorff and box-counting dimensions. In this paper, we study the dimensions and the intermediate dimension spectra of  non-autonomous conformal fractals generated by non-autonomous iterated function systems.

\subsection{Iterated function systems.} \label{subsec_IFS}
Let $\{\varphi_{i}\}_{i\in I} $ be a finite set of contractions on
$\R^{d}$ with $\#I\geq 2$. The set $\{\varphi_{i}\}_{i\in I}$ is called an {\em   iterated function system(IFS)}, and it has a unique
{\em attractor}, that is, a unique non-empty compact $E \subset \R^{d}$
such that
\begin{equation}\label{saattractor}
E = \bigcup_{i\in I} \varphi_{i}(E).
\end{equation}
If the contractions $\varphi_i$ are conformal $($similar or affine$)$ mappings, we call  $\{\varphi_{i}\}_{i\in I}$  a {\em   conformal $($similar or affine$)$  iterated function system}, and call $E$ a {\em self-conformal $($self-similar or self-affine$)$ set}. Note that self-similar sets are special cases of self-conformal and self-affine sets. See~\cite{Fal,  Wen00} for details.

Formulae giving the Hausdorff and box-counting dimensions of self-similar sets satisfying the open set condition are well-known. However, calculation of the dimensions of self-conformal and self-affine sets is more awkward. In \cite{DM}, Mauldin and Urbański studied self-conformal sets $E$, and defined a critical value $h$ by topological pressure $P(t)$ satisfying $h=\inf\{t: \overline P(t)<0\}=\sup\{t: \overline P(t)>0\}$. They obtained the following
\begin{equation} \label{eq_sch}
\hdd E=\ubd E=h.
\end{equation}
Moreover, they  extended  to  infinite conformal iterated function systems allowing $\# I=\infty$.  By approximating with finite subsystems,  they obtained $\hdd E=h$. In \cite{DM_1}, Mauldin and Urbański further studied the box and packing dimensions of infinite conformal iterated function systems.  Falconer in  \cite{Falco88} studied the dimensions of self-affine sets. He defined a critical value $d(T_1,\ldots,T_{\# I})$ by using  singular value functions, which is often called  \textit{affine dimension} or \textit{Falconer dimension}, and under certain conditions, he proved
\begin{equation}\label{sadimf}
\hdd E = \dim_B  E=\min \{d, d(T_1, \ldots, T_{\# I})\},
\end{equation}
almost surely.   We refer readers to~\cite{BHR19,Falco92,Falco05, Feng23, HocRap,  MorShm} for various related studies on iterated function systems.

Recently,  Banaji and Fraser in \cite{BJ1} studied the intermediate dimensions of infinite conformal sets and obtained that
$$
\uid E=\max\{h, \uid P\}  \qquad  \textit{for }\theta\in [0,1],
$$
where $P$ is any subset of $J$ that intersects $\varphi_i(J)$ in exactly one point for each $i\in\mathbb{N}$, and $h$ is given by \eqref{eq_sch}.

In this paper, we study the   fractal dimensions and dimension spectra of non-autonomous finite conformal sets, and  we extend our results to   non-autonomous infinite conformal iterated function systems.

	\subsection{Non-autonomous iterated function systems.} \label{subsecNIFS}

Non-autonomous iterated function systems may be regarded as a generalization of iterated function systems. In 1946, Moran first studied non-autonomous similar sets, also called Moran sets which is a special class of fractals generated by non-autonomous iterated function systems; see \cite{GM,MP}. First, we recall the definitions of  non-autonomous iterated function systems.

Let $\{I_{k}\}_{k\geq 1}$ be a sequence of index sets with $\# I_k\geq 2$ which is either finite or countable infinity. Given integers $k\ge l \ge 1$,  we write
\begin{equation}\label{finite I}
\Sigma_l^{k}=\{u_{l}u_{l+1}\ldots u_{k}: u_{j}\in I_j, j=l,l+1,\ldots, k \},
\end{equation}
and for simplicity, we set $ \Sigma^k=\Sigma^{k}_1 $ if $l=1$. We write
$$
\Sigma^{*}=\bigcup _{k=0}^{\infty }\Sigma^{k}
$$
for the set of all finite words with $
\Sigma^{0}=\{\emptyset \}$ containing only the empty word $\emptyset$.

Let $J\subset\mathbb{R}^d$  be a compact set with non-empty interior, and $J$ satisfies $\overline{\mbox{int}(J)}=J$. For each integer $k>0$, let $\Phi_k=\{\varphi_{k, i} \}_{i\in I_k}$ be a family of mappings $\varphi_{k, i}:J\to J$. We say the collection $\mathcal{J}=\{J_{\mathbf{u}}:\mathbf{u}\in \Sigma^*\}$ of closed subsets of $J$ fulfills the \textit{non-autonomous structure with respect to $\{\Phi_k\}_{k=1}^\infty$} if it satisfies the following conditions:
\begin{itemize}
\item[(i).]{\it Uniform contraction}: There exists $0<c<1$ such that for all integers $k>0$ and  all $i\in I_k$, 
\begin{equation} \label{def_uccdn}
|\varphi_{k, i}(x)-\varphi_{k, i}(y)|\le  c|x-y| \qquad \textit{ for all }  x, y\in J.
\end{equation}
\item[(ii).] For all integers $k>0$ and all $\mathbf{u}\in \Sigma^{k-1}$, the elements $J_{\mathbf{u}i}, i\in I_k$ of $\mathcal{J}$ are  subsets of $J_{\mathbf{u}}$. We write $J_{\emptyset }=J$ for the empty word $\emptyset $.

\item[(iii).] For each $\mathbf{u}=u_1\ldots u_k \in \Sigma^*$,  there exists $\w_{\bu} \in \R^d$ and $\Psi_\mathbf{u}: \mathbb{R}^{d}\rightarrow \mathbb{R}^{d}$  such that
\begin{equation}\label{basic set}
    J_{\mathbf{u}}=\Psi_{\mathbf{u}}(J)=\varphi_\bu(J)+\omega_\bu,
   \end{equation}
    where $\varphi_\bu=\varphi_{u_1}\circ \cdots \circ \varphi_{u_j} \cdots \circ \varphi_{u_k} (J)$ and $\varphi_{u_j}\in \Phi_j$.

\end{itemize}
We call  $\Phi=\{\Phi_k\}_{k=1}^\infty$   a \emph{non-autonomous iterated function system} and
\begin{equation}\label{att}
E=E(\Phi)=\bigcap_{k=1}^\infty \bigcup_{\mathbf{u}\in\Sigma^k}\Psi_{\mathbf{u}}(J)
\end{equation}
the \emph{non-autonomous attractor} of $\Phi$.

If $\#I_k<\infty$ for every $k\in\mathbb{N}_+$,  we say $\Phi$ is a \emph{non-autonomous finite iterated function system}(NFIFS) and $E$ is the \emph{non-autonomous finite set} of $\Phi$.      If $\#I_k=\infty$ for infinitely many  $k\in\mathbb{N}_+$,  we say $\Phi$ is a \emph{non-autonomous infinite iterated function system}(NIIFS) and $E$ is the \emph{non-autonomous infinite set} of $\Phi$. Since the non-autonomous infinite set $E$  of $\Phi$  is generated by iterations of contractions, without loss of generality, we assume that  $\#I_k=\infty$ for every  $k\in\mathbb{N}_+$  in $\Phi$.  We do not consider the case where  $\#I_k=\infty$ for only  finitely many  $k\in\mathbb{N}_+$.  Note that the attractors of non-autonomous infinite iterated function systems  may not be compact.

If for all integers $k\ge0$ and $\mathbf{u}\in\Sigma^{k},$ we have
$$
\mbox{int}(J_{\mathbf{u}i})\cap \mbox{int}(J_{\mathbf{u}j})=\emptyset,
$$
for all $i\not= j\in I_{k+1}$, we say that $\Phi$ (or $E$) satisfies the \emph{open set condition (OSC)}.
If  for all integers $k\ge0$ and $\mathbf{u}\in\Sigma^{k},$ we have that
$$
J_{\mathbf{u}i}\cap J_{\mathbf{u}j}=\emptyset,
$$
for all $i\not= j\in I_{k+1}$, we say that $\Phi$ (or $E$) satisfies the \emph{strong separation condition (SSC)}.

If for each integer $k>0$, every mapping $\varphi_{k, j}$ in $\Phi_k$ is a similarity with contraction ratio $c_{k, j}$, and $\Phi$ satisfies OSC, then the attractor $E$ is called a \emph{Moran set}, which was first studied by Moran \cite{MP} in 1946. In 1994, Hua first studied the Hausdorff dimension of Moran sets,  and later, Hua and others showed that if
$$
\lim_{k\to +\infty} \frac{\log{\underline{c}_k}}{\log{M_k}}=0,
$$
where $\underline{c}_k=\min_{1\le j \le \# I_k} \{c_{k,j}\}$ and $M_k=\max_{\mathbf{u} \in \Sigma^k} \lvert J_\mathbf{u} \rvert$, then
\begin{equation} \label{MSDF}
\hdd E=s_{\ast }=\liminf_{k\rightarrow \infty }s_{k}, \quad \pkd E=\ubd E=s^{\ast }=\limsup_{k\rightarrow \infty }s_{k},
\end{equation}
where $s_{k}$ is the unique real solution to $\prod\nolimits_{i=1}^{k}\left(\sum\nolimits_{j=1}^{\# I_i}(c_{i,j})^{s}\right)=1 .$
We refer readers to  \cite{Hua, HL96,   HRWW, Wen00} for details.  There are various generalizations of such fractals. In \cite{GM}, Gu and Miao studied the general theory of non-autonomous iterated function systems. In \cite{HZ}, Holland and Zhang studied the Hausdorff dimension of a class of special non-autonomous iterated function systems and the box and packing dimension of a class of generalized non-autonomous conformal sets.  We refer readers to  \cite{GM22, GHM,LLMX, Wen00} for other related works.

	\subsection{Non-autonomous conformal iterated function systems.} \label{subsec_NCIFS}
In this paper, we mainly focus on the dimension theory of non-autonomous conformal sets.

 Let $V\subset \mathbb{R}^d$ be an open set and $\phi:V \to V$ be a conformal map.  We denote by  $D\phi(x)$ the
 derivative of $\phi$  at $x$, i.e., $D\phi(x): \R^d\to \R^d  $ is a similarity linear map, and
 we denote by  $\|D\phi(x)\| $ the scaling factor at $x$. For non-autonomous conformal systems, we write
$$
 \|D\phi\| =\sup\{\|D\phi(x)\| : x \in J\}.
$$   
Let $\Phi$ be a non-autonomous finite (or infinite) iterated function system satisfying that
\begin{itemize}
\item[(i').]{\it Conformality}: There exists an open connected set $V$ independent of $k$ with $J\subset V$ such that each $\varphi_{k, j}$ extends to a $C^1$ conformal diffeomorphism of $V$ into $V$.
\item[(ii').]{\it Bounded distortion}: There exists a constant $C\ge 1$ such that for all $\bu=u_lu_{l+1}\ldots u_k\in\Sigma_l^k$, the map $\varphi_\mathbf{u}=\varphi_{l, u_l}\circ \varphi_{l+1, u_{l+1}}\circ \dots\circ \varphi_{k, u_k}$ satisfies
$$
\|D\varphi_{\bu}(x)\|\le C\|D\varphi_{\bu}(y)\|
$$
for all $x, y\in V$.
\end{itemize}
We say $\Phi=\{\Phi_k\}_{k=1}^\infty$ is a \emph{non-autonomous finite (or infinite) conformal iterated function system(NFCIFS or NICIFS)}, and we call
\begin{equation}\label{att}
E=E(\Phi)=\bigcap_{k=1}^\infty \bigcup_{\mathbf{u}\in\Sigma^k}J_\bu
\end{equation}
is the \emph{ non-autonomous finite (or infinite) conformal set} of $\Phi$.

Let $\Phi$ be the NFCIFS and $E$  the corresponding attractor given by \eqref{att}.   Let
\begin{equation}\label{Mk}
\begin{split}
&m_k=\inf_{\mathbf{u} \in \Sigma^k} \{\|D\varphi_\mathbf{u}\|\}, \qquad \quad M_k=\sup_{\mathbf{u} \in \Sigma^k} \{\|D\varphi_\mathbf{u}\|\}, \\
&\underline c_k=\inf_{1\le j \le \# I_k} \{\|D\varphi_{k, j}\|\}, \quad  \quad \overline c_k=\sup_{1\le j \le \# I_k} \{\|D\varphi_{k, j}\|\}.
\end{split}
\end{equation}
In this paper, we study the non-autonomous finite conformal iterated function systems  $\Phi$ satisfying open set condition and
\begin{equation}\label{condition}
\lim_{k\to +\infty} \frac{\log\underline c_k}{\log M_k}=0,
\end{equation}
which plays a fundamental role in the study  of Moran fractals; see \cite{Wen00} for details.  In \cite{RM}, Rempe-Gillen and Urbański studied Hausdorff dimensions for non-autonomous conformal sets under the different assumption
\begin{equation}\label{cdn_RU}
\lim_{k\to\infty}\frac{\log \#I_k}{k}=0,
\end{equation}
and they provided  the Hausdorff dimension  formula for such  sets.  However, the conditions \eqref{condition} and \eqref{cdn_RU} do not imply each other;  see Example \ref{ex_nec}.
   Inspiring by their work,  we  further relax  the condition \eqref{condition}  to
\begin{equation}\label{def_cmc}
\lim_{k\to +\infty} \frac{\log\overline c_k-\log\#I_k}{\log M_k}=0.
\end{equation}
See Proposition \ref{lb packing} and Corollary  \ref{thm_Hdim g} for details. Under the cone condition, we make certain generalization of the conclusions in \cite{RM}; see  Corollary \ref{mphi} and Corollary \ref{Cor_RUI}.

Next, we provide an example to illustrate the difference to  condition \eqref{condition}, condition \eqref{def_cmc} and condition \eqref{cdn_RU}.
\begin{exmp} \label{ex_nec}
Let  $E_1$ be a   Moran set with $\#I_k =2^k$ and $c_{k,j} =\frac{1}{3^{k+1}}$. It is clear that
$$
\lim_{k\to +\infty} \frac{\log\underline c_k}{\log M_k}=0 \quad \textit{ and} \quad  \lim_{k\to\infty}\frac{\log \#I_k}{k}=\log 2.
$$
Then by \eqref{MSDF},  we have  $\hdd E_1=\frac{\log 2}{\log 3}.$

Let  $E_2$ be a Moran set with $\#I_k =2^k$ ,$c_{k,1}={\frac{1}{3^{k+1}}}$ and $c_{k,j}={\frac{1}{3^{k(k+1)}}}  (2\le j\le 2^k)$. It is straightforward that
$$
\lim_{k\to +\infty} \frac{\log\underline c_k}{\log M_k} =2, \quad \qquad \lim_{k\to\infty}\frac{\log \#I_k}{k}=\log 2,  \qquad  \quad
\lim_{k\to +\infty} \frac{\log\overline c_k-\log\#I_k}{\log M_k} =0.
$$
By Theorem \ref{thm_Hdim g}, we have $\hdd E_2=0$.

Let $E_3$ be a homogeneous Moran set with $\#I_k =2$ ,$c_{1}={\frac{1}{2}}$, $c_{2}={\frac{1}{4}}$ and $c_{k}=c_1c_2\cdots c_{k-1} (k\ge 3)$. Then it follows that
$$
\lim_{k\to\infty}\frac{\log \#I_k}{k}=0  \quad \textit{ and} \quad \lim_{k\to +\infty} \frac{\log\overline c_k-\log\#I_k}{\log M_k} =\frac{1}{2}.
$$
By simple calculation, we have $\hdd E_3=0$.
\end{exmp}

\section{Main conclusions}

Let $E$ be the non-autonomous  conformal set of the non-autonomous finite iterated function system  $\Phi$, that is, $\# I_k <\infty$ for all $k>0$. First,  we give  intermediate dimension spectra formula for   $E$.
For real $\delta>0$ and  $\theta\in(0,1]$, let
$$
\mathcal{M}(\delta, \theta)=\{\mathcal{M}:\mathcal{M} \textit{ is a cut set satisfying }\delta^\frac{1}{\theta} < \|D\varphi_{\mathbf{u}^*}\| \textit{ and }  \|D\varphi_\mathbf{u}\|\le \delta \textit { for each } \mathbf{u} \in \mathcal{M}\},
$$
and we write
$$
k_\delta=\min\{|\mathbf{u}| : \|D\varphi_\mathbf{u}\|\le\delta< \|D\varphi_{\mathbf{u}^*}\|,  \mathbf{u} \in \Sigma^* \}.
$$
For $\theta\in(0,1]$, upper and lower pressure functions are given by
\begin{equation}\label{def_PTF}
\begin{split}
\overline P(t, \theta) &= \limsup_{\delta\to 0} \frac{1}{k_{\delta}}\log \min_{\mathcal{M}\in\mathcal{M}(\delta, \theta)}\{\sum_{\mathbf{u}\in\mathcal{M}}\|D\varphi_\mathbf{u}\|^t\},  \\
\underline P(t, \theta) & = \liminf_{\delta\to 0} \frac{1}{k_{\delta}}\log \min_{\mathcal{M}\in\mathcal{M}(\delta, \theta)}\{\sum_{\mathbf{u}\in\mathcal{M}}\|D\varphi_\mathbf{u}\|^t\}).
\end{split}
\end{equation}
We write their jump points respectively as
\begin{equation}\label{def_stheta}
\begin{split}
s^\theta &=\inf\{t:\overline P(t, \theta)<0\}=\sup\{t:\overline P(t, \theta)>0\}, \\
s_\theta &=\inf\{t:\underline P(t, \theta)<0\}=\sup\{t:\underline P(t, \theta)>0\}.
\end{split}
\end{equation}
Note that  the existence of $s^\theta$ and $s_\theta$ follows from the monotonicity of $\underline P(t)$ and $\overline P(t)$; see  Lemma \ref{ddx1} in section \ref{sec_PF}.
When $\theta=0$, we write
$$
s_0=s^0=\hdd E.
$$

\begin{thm}\label{thm_IMdim}
Let $E$ be the non-autonomous finite conformal set  satisfying \eqref{condition} and OSC.
Then upper and lower intermediate dimension of $E$  are given by
$$
\uid E=s^\theta, \qquad
\lid E=s_\theta     \qquad  \textit{for } \theta\in[0,1],
$$
where $s^\theta$ and $s_\theta$ are given by \eqref{def_stheta}.
\end{thm}

The box and Hausdorff dimensions of non-autonomous finite conformal sets may be given in a simpler way. For non-autonomous finite (or infinite) conformal iterated function systems, we  define upper and lower pressures by
\begin{equation}\label{def_PFT2}
\overline P(t)=\limsup_{k\to \infty} \frac{1}{k}\log \sum_{\mathbf{u}\in \Sigma^k}\|D\varphi_\mathbf{u}\|^t,\qquad
\underline P(t)=\lim\inf_{k\to \infty} \frac{1}{k}\log \sum_{\mathbf{u}\in \Sigma^k}\|D\varphi_\mathbf{u}\|^t.
\end{equation}
We write their jump points respectively as
\begin{equation}\label{def_s*}
\begin{split}
s^* &=\inf\{t: \overline P(t)<0\}=\sup\{t: \overline P(t)>0\}, \\
s_* &=\inf\{t: \underline P(t)<0\}=\sup\{t: \underline P(t)>0\}.
\end{split}
\end{equation}
The next conclusion shows that the critical value $s^*$ always gives the box and packing dimensions  of non-autonomous finite conformal sets.
\begin{thm}\label{thm_PuBdim}
Let $E$ be the non-autonomous finite conformal set   satisfying \eqref{condition} and OSC.
Then the upper box and packing dimensions  of $E$ are given by
$$
\pkd E=\ubd E=s^*,
$$
where $s^*$ is given by \eqref{def_s*}.
\end{thm}
In Section \ref{sec_PAB}, we change the condition \eqref{condition} in Theorem \ref{thm_PuBdim} to \eqref{def_cmc}, and obtain  some partial results. See Proposition \ref{lb packing}.

Generally, Hausdorff dimensions of non-autonomous finite conformal sets are not easy to calculate. We obtain Hausdorff dimensions under extra conditions.

\begin{thm}\label{thm_Hdim}
Let $E$ be the non-autonomous finite conformal set with $\mathcal{L}^d(\partial J)=0$ satisfying \eqref{condition} and OSC. Then the Hausdorff dimension is given by
$$
\hdd E=s_*,
$$
where $s_*$ is given by \eqref{def_s*}.
\end{thm}

Next,  we weaken the condition \eqref{condition} in Theorem \ref{thm_Hdim} to \eqref{def_cmc}.
\begin{cor}\label{thm_Hdim g}
Let  $E$ be a  non-autonomous finite conformal set with $\mathcal{L}^d(\partial J)=0$  satisfying \eqref{def_cmc} and OSC. Then the Hausdorff dimension of $E$ is given by
$$
\hdd E=s_*,
$$
where $s_*$ is given by \eqref{def_s*}.
\end{cor}

\begin{rem}
If $J$ is a convex set,  $\mathcal{L}^d(\partial J)=0$ holds, for example  $J=[0,1]^d$.
\end{rem}

\begin{rem} If $E$ satisfies the strong separation condition, condition $\mathcal{L}^d(\partial J)=0$ may be omitted.
\end{rem}

Non-autonomous infinite conformal sets are more awkward to study, and we extend our conclusions to  such sets with the following assumption instead of \eqref{condition},
\begin{equation}\label{cdn2}
\lim_{k\to +\infty} \frac{\log\overline c_k}{\log M_k}=0.
\end{equation}

For a NICIFS $\Phi$, let
$$
M_k=\sup_{\mathbf{u} \in \Sigma^k} \{\|D\varphi_\mathbf{u}\|\},  \quad \overline c_k=\sup_{j \in I_k} \{\|D\varphi_{k, j}\|\}.
$$  
\begin{thm}\label{thm_infinite}
Let $E$ be the non-autonomous infinite conformal set  satisfying \eqref{cdn2} and OSC.
Suppose for all $t\in(0,d)$ and all $\epsilon>0$, the following hold
\begin{itemize}
\item[(1).]The sums $\sum_{j\in I_k}\|D\varphi_{k, j}\|^t$ are either infinite for all $k$ or finite for all $k$.
\item[(2).] If $\sum_{j\in I_k}\|D\varphi_{k, j}\|^t<\infty$, then
$$
\lim_{n\to\infty}\frac{\sum_{k\le {M_n^{-\epsilon}}}\|D\varphi_{n, k}\|^t}{\sum_{j\in I_n}\|D\varphi_{n, j}\|^t}=1.
$$
\item[(3).]If $\sum_{j\in I_k}\|D\varphi_{k, j}\|^t=\infty$, then $\lim_{n\to\infty}\sum_{k\le {M_n^{-\epsilon}}}\|D\varphi_{n, k}\|^t=\infty. $
\end{itemize}
Let $s^*$ and  $s_*$ be given by \eqref{def_s*}. Then  $\ubd E\ge\pkd E\ge s^*$. Furthermore, if $\mathcal{L}^d(\partial J)=0$,
then  $\hdd E=s_*$.
\end{thm}

 The rest of the paper is organized as follows. In Section \ref{sec_nagIFS},  we study the general properties of non-autonomous IFS and estimate the dimensions of non-autonomous  fractals. In Section \ref{sec_PF}, we study the properties of pressure functions, which are essential to study the dimensions of non-autonomous conformal sets. In Section \ref{sec_ID}, we give the proof of Theorem \ref{thm_IMdim}. Theorem \ref{thm_PuBdim} is proved in Section \ref{sec_PAB}.  In Section \ref{sec_HF}, we prove  Theorem \ref{thm_Hdim}, and relax the condition \eqref{condition} to \eqref{def_cmc}.   Finally, we extend our conclusions to non-autonomous infinite conformal fractals in Section \ref{sec_InfIFS}.

\section{Dimension estimate of non-autonomous sets}\label{sec_nagIFS}
In this section, we study the properties of non-autonomous iterated function systems  defined in Subsection \ref{subsecNIFS}, and these properties are  useful to explore  the  non-autonomous conformal fractals.
\subsection{Symbolic space and pressure functions}
We write
$$
\Sigma^{\infty}=\{\mathbf{u}=u_{1}u_{2}\ldots u_{k}\ldots : u_{k}\in I_k \, k=1,2,\ldots \}
$$
for the set of words with infinity length, and we  topologize $\Sigma^\infty$ using the metric
$d(\mathbf{u},\mathbf{v})=2^{-|\mathbf{u}\wedge  \mathbf{v}|}$ for distinct $\mathbf{u},\mathbf{v} \in \Sigma^\infty$ to make $\Sigma^\infty$ into a compact metric space. Given an integer $n>0$, for every $\bu\in \Sigma^\infty$, we write $\bu|_n=u_1\ldots u_n$.
For each $\mathbf{u}=u_1 \dots u_k \in \Sigma^k$, we write $\mathbf{u}^* =u_1 \dots u_{k-1}$.  Given $\mathbf{u}\in \Sigma^l$, for $\mathbf{v}\in \Sigma^k$ where $k\geq l$ or $\mathbf{v}\in\Sigma^\infty $,  we write $\mathbf{u}\prec \mathbf{v}$ if $u_i =v_i$ for all $i=1,2,\ldots, l$.
	
We define the \textit{cylinders} $\mathcal{C}_\mathbf{u}=\{\mathbf{v}\in \Sigma^\infty : \mathbf{u}\prec \mathbf{v}\}$ for $\mathbf{u}\in \Sigma^*$; the set of cylinders $\{\mathcal{C}_\mathbf{u} : \mathbf{u} \in \Sigma^* \}$ forms a base of open and closed neighborhoods for $\Sigma^\infty$. We term a subset $A$ of $\Sigma^*$ a
\textit{cut set} if $\Sigma^\infty\subset\bigcup_{\mathbf{u}\in A}\mathcal{C}_\mathbf{u}$, where $\mathcal{C}_\mathbf{u}\bigcap\mathcal{C}_{\mathbf{v}}=\emptyset$ for  all $\mathbf{u}\neq \mathbf{v}\in A$. It is equivalent to that, for every $\mathbf{w}\in \Sigma^\infty$, there is a unique word $\mathbf{u}\in A$ with $|\mathbf{u}|<\infty$ such that $\mathbf{u}\prec \mathbf{w}$.

Let $\pi_\Phi:\Sigma^\infty\to J$ by
$$
 \pi_\Phi(\mathbf{u})=\bigcap_{n=1}^\infty J_{\mathbf{u}|_n},
$$
where $J_{\mathbf{u}|_n}$ is given by \eqref{basic set}.

We cite K\"{o}nig's lemma in graph theory, and see \cite{WR} for details.
\begin{lem}\label{Konig's}

Let $T$ be a subset of $\Sigma^*$ such that  $\#T=\infty$ and  $\emptyset\in T$. Suppose that   for each $\bu\in T$,  the set $Pr_T(\bu)=\{\bv\in T:\bv\prec \bu\}$ and  the set
$$
succ_T(\bu)=\{\bv\in T:\bu\prec \bv \textit{ and } |\bv|=|\bu|+1\}
$$
are  finite. Then there exists  $\mathbf{w}\in\Sigma^\infty$ such that  $\mathbf{w}|_n\in T$ for each $n>0$.
\end{lem}

Given a non-autonomous iterated function system $\Phi$, we say that $\Phi$ satisfies the finite overlap condition (FOC) if for all $\bu\in\Sigma^*$ and all $x\in J_\bu$, we have
$$
\#\{i\in I_{|\bu|+1}:x\in J_{\bu i}\}<\infty.
$$

Next, under the finite overlap condition, we use  $\pi_\Phi$ to give another representation of  non-autonomous attractors.
\begin{lem}\label{Sigma=E}
If a non-autonomous  iterated function system $\Phi$ satisfies the finite overlap condition, then the attractor $E$ of $\Phi$ is the image of $\pi$, that is 
$$
E=\pi_\Phi(\Sigma^\infty)=\bigcap_{k=1}^\infty \bigcup_{\mathbf{u}\in\Sigma^k}J_\bu.
$$      
\end{lem}

\begin{proof}
For each $\bu\in\Sigma^\infty$, we have $\pi_\Phi(\mathbf{u})=\bigcap_{n=1}^\infty J_{\mathbf{u}|_n}$. Since $J_{\mathbf{u}|_n}\subset\bigcup_{\mathbf{v}\in\Sigma^n}J_\bv, $
it follows that $\pi_\Phi(\mathbf{u})\in E$ and $\pi_\Phi(\Sigma^\infty)\subset E$.

For each $x\in E$, let $T=\Sigma(x)=\{\bu\in\Sigma^*:x\in J_\bu\}.$ Since  for each $n\in\mathbb{N}$, there exists  $\bu\in\Sigma^n$ such that $x\in J_\bu$, we have $\#T=\infty$. For each $\bu\in T$, it is clear that $Pr_T(\bu)=\{\bu|_k:1\le k\le |\bu|\}$ is finite. By  finite overlap condition, we have that  $\# succ_T(\bu)=\#\{i\in I_{|\bu|+1}:x\in J_{\bu i}\}<\infty,$ and by Lemma \ref{Konig's},  there exists  $\mathbf{w}\in\Sigma^\infty$ such that  $\mathbf{w}|_n\in T$ for each $n>0$. It implies that  $x\in J_{\mathbf{w}|_n}$ for every $n\in\mathbb{N}$ and $x=\pi_\Phi(\mathbf{w})$, and we have $E\subset \pi_\Phi(\Sigma^\infty)$.
\end{proof}

\begin{rem}
If $\Phi$ is   a non-autonomous finite iterated function system, the finite overlap condition automatically  holds, but the finite overlap condition is not necessarily satisfied  for a non-autonomous infinite iterated function system.
\end{rem}

Let $\Phi$ be the non-autonomous finite or infinite iterated function system satisfying OSC, and $E$  the attractor of $\Phi$ given by \eqref{att}. Given $k\ge l\ge0$, for each $\bu\in\Sigma_l^k$,   let
$$
\overline c_\bu=\sup_{x, y\in J}\frac{|\varphi_\bu(x)-\varphi_\bu(y)|}{|x-y|} \quad \textit{ and} \quad
\underline c_\bu=\inf_{x, y\in J}\frac{|\varphi_\bu(x)-\varphi_\bu(y)|}{|x-y|}.
$$
It is clear that
\begin{equation}\label{r guji}
\underline c_\bu|x-y|\le|\varphi_\bu(x)-\varphi_\bu(y)|\le\overline c_\bu|x-y|.
\end{equation}

We  define upper and lower pressures respectively by
\begin{equation} \label{def_pf}
\overline P(t)=\limsup_{k\to \infty} \frac{1}{k}\log \sum_{\mathbf{u}\in \Sigma^k}\overline c_\bu^t \quad \textit{ and} \quad
\underline P(t)=\lim\inf_{k\to \infty} \frac{1}{k}\log \sum_{\mathbf{u}\in \Sigma^k}\overline c_\bu^t,
\end{equation}
and both of them are strictly monotonous.
\begin{lem}\label{lem_ddx}
For  all $t_1 <t_2$,  if $\overline P(t_1)$ and $\overline P(t_2)$ are finite, then  $\overline P(t_1) > \overline P(t_2)$.
\end{lem}

\begin{proof}
Given $t_1<t_2$ such that $\overline P(t_1)$ and $\overline P(t_2)$ are finite. By \eqref{def_uccdn}, we have
$$
\sum_{\mathbf{u}\in \Sigma^n}\overline c_\bu^{t_2}  \le c^{n(t_2-t_1)}\sum_{\mathbf{u}\in \Sigma^n}\overline c_\bu^{t_1} ,
$$
where $c$ is given by \eqref{def_uccdn}, and it follows that $\overline P(t_2)\le\overline P(t_1)-(t_2-t_1)\log\frac{1}{c}.$
Hence  we have $\overline P(t_2) < \overline P(t_1)$ since $0<c<1$.
\end{proof}

\subsection{Dimensions estimation of non-autonomous fractals }
Let $ \overline P(t)$ and $\underline P(t)$ be the pressure functions given by \eqref{def_pf}.  By Lemma \ref{lem_ddx},  we write their jump points respectively as
\begin{eqnarray}\label{s^*}
\begin{split}
s^*&=&\inf\{t: \overline P(t)<0\}=\sup\{t: \overline P(t)>0\},  \\
\label{s_*}
s_*&=&\inf\{t: \underline P(t)<0\}=\sup\{t: \underline P(t)>0\}.
\end{split}
\end{eqnarray}
Since these are  consistent with the critical values defined in \eqref{def_s*},  we use the same notation for simplicity.
Let
\begin{eqnarray}\label{box ub}
t^*=\inf\{t:\sum_{k=1}^\infty\sum_{\mathbf{u}\in\Sigma^k}\overline c_\mathbf{u}^t<\infty\}.
\end{eqnarray}

Generally, $s_*$  serves as an upper bound for the Hausdorff dimension of non-autonomous attractors,  $t^*$ as an upper bound for the box dimension of non-autonomous attractors, and  $s^*$ as a lower bound for the box dimension of non-autonomous attractors. For some  special cases, we may  have $s^*=t^*$, and  see Lemma \ref{box d} for details.

\begin{thm}\label{ub hausdorff}
Given a non-autonomous finite or infinite iterated function system $\Phi$ satisfying OSC, let  $E$ be the corresponding  non-autonomous set given by \eqref{att} and $s_*$  given by \eqref{s^*}. Then
$$
\hdd E\le s_*.
$$
If $\# I_k $ is finite for all $k>0$,  and  there exists a constant  $N>0$ such that
\begin{equation}\label{cdn_uNk}
\limsup_{k\to\infty}\frac{\log \# I_k}{\log (-\log\max_{\bu\in\Sigma^{k-1}}\{\overline c_\bu\})}\le N,
\end{equation}
 then
$$
\ubd\le t^*,
$$
where $t^*$ is given by \eqref{box ub}.
\end{thm}

\begin{proof}
For every $\bu\in\Sigma^*$, by \eqref{r guji},  we have $|J_\bu|\le \overline c_\bu |J|.$ For each real $t>s_*$, we have $\underline P(t)<0$, and by \eqref{def_pf},  there exists $\{n_k\}$ such that for each $k$,
$$
|J|^{-t}\sum_{\mathbf{u}\in\Sigma^{n_k}}|J_\bu|^t\le\sum_{\mathbf{u}\in\Sigma^{n_k}}\overline c_\bu^t<e^{\frac{\underline P(t)}{2}n_k}.
$$
Let $\delta_k=\max\{|J_\mathbf{u}|:\mathbf{u}\in\Sigma^k\}$. Thus  the Pre-Hausdorff measure of $E$ is bounded by
$$
\mathcal{H}^t_{\delta_{n_k}}(E)\le |J|^{t}e^{\frac{\underline P(t)}{2}n_k},
$$
and the Hausdorff measure of $E$ is  zero.  Hence,   $\hdd E\le t$ for all $t>s_*$, and we obtain that $\hdd E\le s_*$.

Suppose that  $\# I_k $ is finite for all $k>0$. Next, we prove $\ubd E\le t^*$. For every $s>t^*$, by \eqref{box ub}, it is clear that
$$
\sum_{k=1}^\infty\sum_{\mathbf{u}\in\Sigma^k}\overline c_\mathbf{u}^s<\infty.
$$
For all sufficiently small  $\delta>0$, we write
$$
\mathcal{M}'(\delta)=\{\bu\in\Sigma^*:\overline c_\bu\le \delta<\overline c_{\bu^*}\}.
$$
Denote by $N_\delta(E)$  the smallest number of sets with diameters at most $\delta$  covering the set $E$, and it is clear that $N_\delta(E)\le\#\mathcal{M}'(\delta)$.

We write        
$$
\mathcal{M}''(\delta)=\{\bu\in\Sigma^*: \textit{ there exists $j$ such that } \bu j\in\mathcal{M}'(\delta)\},
$$
and note that $\mathcal{M}''(\delta)$ may not be a cut set. Let
\begin{equation}\label{def_ndelta}
N(\delta)=\max\{\#I_k:  \mathcal{M}'(\delta)\cap \Sigma^k \neq \emptyset \textit{ for all }k\in\mathbb{N}_+\}.
\end{equation}                         
Then we have $ \#\mathcal{M}'(\delta)\le N(\delta)\#\mathcal{M}''(\delta) ,$ and it follows that
$$
N_\delta(E)\delta^s\le N(\delta)\#\mathcal{M}''(\delta) \delta^s\le N(\delta) \sum_{\bu\in\mathcal{M}''(\delta)}\overline c_\bu^s\le N(\delta)\sum_{k=1}^\infty\sum_{\mathbf{u}\in\Sigma^k}\overline c_\mathbf{u}^s.
$$
Thus
$$
\frac{\log N_\delta(E)}{-\log \delta}\le s+\frac{\log\sum_{k=1}^\infty\sum_{\mathbf{u}\in\Sigma^k}\overline c_\mathbf{u}^s}{-\log \delta} +\frac{\log N(\delta)}{-\log \delta}.
$$
Since  $\delta<\max_{\bv\in\Sigma^{|\bu|-1}}\{\overline c_\bv\}$ for every $\bu\in\mathcal{M}'(\delta)$,  by \eqref{cdn_uNk}, we obtain that
$$
\limsup_{\delta\to 0}\frac{\log \# I_{|\bu|}}{-\log\delta}<\limsup_{\delta\to 0}\frac{\log \# I_{|\bu|} }{\log (-\log\max_{\bv\in\Sigma^{|\bu|-1}} \{\overline c_\bv\}) } \frac{\log (-\log\delta)}{(-\log\delta)}= 0.
$$
By \eqref{def_ndelta}, this implies that $\limsup_{\delta\to 0}\frac{\log N(\delta)}{-\log\delta} =0$.
Therefore,  $\ubd E\le s$ for all $s>t^*$, and the conclusion holds.
\end{proof}

\subsection{Non-autonomous infinite iterated function systems}\label{subsec IIFS}
Given a non-autonomous iterated function system $\Phi=\{\Phi_k\}_{k=1}^\infty$. 
In this subsection, we assume that there exists a constant $D>0$ such that  for all $k\ge l\ge0$ and all  $\bu\in\Sigma_l^k$,
\begin{equation}\label{cdn_cucl}
\frac{\overline c_\bu}{\underline c_\bu}\le D,
\end{equation}
where $ \overline c_\bu$ and $ \underline c_\bu$ are given by \eqref{r guji}. We say $\Phi'= \{\Phi_k'\}_{k=1}^\infty$ is a \emph{subsystem of $\Phi$} if $I'_k\subset I_k$ and $\Phi_k' \subset \Phi_k$ for all $k>0$, and we  write $\overline P'(t)$ and  $\underline P'(t)$ for the upper and lower pressure functions of the subsystem  $\Phi'$, and write $(\Sigma_l^k)'$ for the subspace of $\Sigma_l^k$. 

Generally, it is difficult to calculate  the dimensions of non-autonomous sets. Therefore, we adopt the method of approximating the infinite system with a finite subsystem to obtain an estimation of the dimension of the non-autonomous infinite iterated function system. The following  conclusion is inspired by \cite{RM}.

\begin{lem}\label{ap by finite}
Let $\Phi$ be a non-autonomous finite or infinite iterated function system satisfying \eqref{cdn_cucl}. Given a real $t\in[0,d]$, assume that  $ \sum_{k\in I_n}\overline c_{n, k}^t<\infty$ for all $n$.
Given  $\delta>0$ and, for each $n$, let $I'_n\subset I_n$ be a finite index set such that
\begin{equation}\label{app}
\sum_{k\in I_n}\overline c_{n, k}^t\le(1+\delta)\sum_{k\in I'_n}\overline c_{n, k}^t,
\end{equation}
for all sufficiently large $n$.
Let $\Phi'$ be the  subsystem  given by $I'_n$.  Then
$$
\overline P'(t)\ge \overline P(t)-\delta D^{2d},    \qquad \underline P'(t)\ge \underline P(t)-\delta D^{2d}.
$$
\end{lem}

\begin{proof}
Fix $n>0$, and  for each $s=0, 1, \ldots   ,n$, we write
$$
\mathcal{P}_s=\{P\subset \{1, \ldots, n\}:\#{P}=s\}.
$$
For every $P\in \mathcal{P}_s$, we divide $P$ into consecutive parts and write
\begin{eqnarray*}
&&I^*_P= \{m (m+1)\cdots (j-1)j(j+1)\cdots l: \textit{ for $m, l \in\{1,\ldots, n \}$} \\
&& \hspace{3cm}  \textit{ s.t. }  m-1\notin P, l+1\notin P \textit{ and } j\in P \textit{ } \forall  j=m,\ldots, l\}
\end{eqnarray*}
and
$$
\Sigma^n(P)=\{\bu=u_1\ldots u_n\in\Sigma^n:u_i\in I'_i \textit{ for }i\in P, u_j\notin I'_i \textit{ for }j\notin P\}.
$$
For each $\bu\in \Sigma^n(P)$, let
$$
\Sigma^n(P, \bu)=\{\bv=\bu|_m^l: m(m+1)\cdots l\in I_P^* \}.
$$     

Let  $k=\# I^*_P$. It is clear that $k-1\le n-s$.
Recall \eqref{r guji}, and by \eqref{cdn_cucl}, it follows that for each $\bu\in\Sigma^n$
$$
\overline c_\bu\le\prod_{\bv\in \Sigma^n(P, \bu)} \overline c_{\bv}\prod_{j\notin P}\overline c_{u_j}  \le D^{k+n-s-1}\overline c_\bu   \le D^{2n-2s}\overline c_\bu.
$$
We write ${\Sigma'}^n =\{u_{1}u_{2}\ldots u_{n}: u_{j}\in I_j', j=1,2,\ldots, n \} $.  By \eqref{app}, this implies that
\begin{eqnarray*}
\sum_{\bu\in \Sigma^n(P)}\overline c_\bu^{t}  &\leq& \sum_{\bu\in \Sigma^n(P)} \Big(\prod_{\bv\in \Sigma^n(P, \bu)} \overline c_{\bv}\prod_{j\notin P}\overline c_{u_j}  \Big)^{t}  \\
&=&  \Big(\sum_{\bu\in \Sigma^n(P)}\prod_{\bv\in \Sigma^n(P, \bu)}\overline c_{\bv}^{t}\Big) \Big(\prod_{j\notin P}\sum_{u_j\in I_j- I_j'}\overline c_{u_j}^t\Big) \\
&\leq& \delta^{n-s} \Big(\sum_{\bu\in \Sigma^n(P)}\prod_{\bv\in \Sigma^n(P, \bu)}\overline c_{\bv}^{t} \Big)\Big(\prod_{j\notin P}\sum_{u_j\in I_j'}\overline c_{u_j}^t\Big) \\
&\le&\delta^{n-s}D^{2t(n-s)}\sum_{\mathbf{u}\in ({\Sigma}^n)'}\overline c_\bu^{t}.
\end{eqnarray*}
Since the number of elements in $\mathcal{P}_s$ is $C_n^s$ and $t\leq d$,  it immediately follows that
\begin{eqnarray*}
\sum_{\mathbf{u}\in \Sigma^n}\overline c_\bu^{t}&=&\sum_{s=0}^n\sum_{P\in\mathcal{P}_s}\sum_{\bu\in \Sigma^n(P)}\overline c_\bu^{t} \\ &\le&\sum_{s=0}^n\sum_{P\in\mathcal{P}_s}\delta^{n-s}D^{2t(n-s)}\sum_{\mathbf{u}\in ({\Sigma}^n)' }\overline c_\bu^{t}      \\
&\le& (1+\delta D^{2d})^n \sum_{\mathbf{u}\in({\Sigma}^n)' }\overline c_\bu^{t}.
\end{eqnarray*}

By \eqref{def_PFT2}, we obtain that 
$$
\underline P'(t)\ge \underline P(t)-\log (1+\delta D^{2d})\ge \underline P(t)-\delta D^{2d}.
$$
Similarly, we have that $\overline P'(t)\ge \overline P(t)-\delta D^{2d}$, and the conclusion holds.
\end{proof}

Under  certain conditions,  a non-autonomous infinite iterated function system  has  a finite subsystem  with  identical pressure functions, and we apply this property to   control the number of contractions in  subsystem. The following  conclusion is inspired by   \cite{RM}.

\begin{prop}\label{lim=0}
Let $E$ be a non-autonomous infinite set  of $\Phi$ satisfying \eqref{cdn_cucl}. Given a real sequence $\{\beta_n\}_{n=1}^\infty$ convergent to  zero. Suppose that $\{\beta_n\}_{n=1}^\infty$ satisfies the following  for all $t\in(0,d)$ and all $\epsilon>0$:
\begin{itemize}
\item[(1).]The sum $\sum_{k\in I_n}\overline c_{n, k}^t$ are either infinite for all $n$ or finite for all $n$.
\item[(2).] If $\sum_{k\in I_n}\overline c_{n, k}^t<\infty$, then
$$
\lim_{n\to\infty}\frac{\sum_{k\le {\beta_n^{-\epsilon}}}\overline c_{n, k}^t}{\sum_{j\in I_n}\overline c_{n, j}^t}=1.
$$
\item[(3).]If $\sum_{k\in I_n}\overline c_{n, k}^t=\infty$, then $\lim_{n\to\infty}\sum_{k\le {\beta_n^{-\epsilon}}}\overline c_{n, k}^t=\infty.$
\end{itemize}
Then there exists a  subsystem $\Phi'$ of $\Phi$ such that $\overline P'(t)=\overline P(t)$ and $\underline P'(t)=\underline P(t)$ for all $t\ge 0$ and
$$
\lim_{k\to\infty}\frac{\log \#I_k'}{\log \beta_n}=0.
$$
Let   $E'$ be the non-autonomous set of $\Phi'$, and  let  $s_* '$ and ${s^*}'$ be given by \eqref{s^*} with respect to $\Phi'$. If $\hdd E'=s_*'$, then $\hdd E=s_* '=s_*.$ If $\pkd E'\ge {s^*}'$, then
$\ubd E\ge \pkd E\ge {s^*} '=s^*.$
\end{prop}

\begin{proof}
 Without loss of generality, we  assume $ \overline{c}_{k, i}\ge\overline{c}_{k, j}$ for all $i\ge j$ and all  $k\in\mathbb{N}$.

Let $h=\min\{d, \sup\{t:\sum_{j\in I_k}\overline c_{k, j}^t=+\infty\}\}$. First,  we consider  $h<d$.    Arbitrarily choose   $t>h$, and by $(1)$, we have that  $\sum_{j\in I_k}\overline c_{k, j}^t$ is finite for all integers $k>0$. 

Fix $\delta\in (0,1)$.  For each integer $i>0$, setting $\epsilon = \frac{1}{i}$, by (2), there exists an integer  $N_i>0$ such that for every $n>N_i$, we have
$$
\frac{\sum_{k\le {\beta_n^{-\frac{1}{i}}}}\overline c_{n, k}^t}{\sum_{k\in I_n}\overline c_{n, k}^t}>\frac{1}{1+\delta}.
$$
We choose $\{N_i\}_{i=1}^\infty$ to be  an increasing sequence.   
Let $I^{t,\delta}_n=\{1, 2, \ldots, [\beta_n^{-\frac{1}{i}}]\}$ for $N_i< n\le N_{i+1}$. Let $\Psi_{t, \delta}$ be the subsystem generated by $I^{t,\delta}_n$, and it is clear that
\begin{equation}\label{subsequence}
\lim_{n\to +\infty} \frac{\log\#(I^{t,\delta}_n)}{\log \beta_n}=0,
\end{equation}
and $(1+\delta)\sum_{k\in I^{t,\delta}_n}\overline c_{n, k}^t>\sum_{k\in I_n}\overline c_{n, k}^t.$
Let $\delta=\frac{1}{j}$. By \eqref{subsequence}, there exists an increasing sequence $\{K_j\}_{j=1}^\infty$ such that for $n>K_j$ we have
$$
\frac{\log\#(I^{t,\frac{1}{j}}_n)}{\log \beta_n}<\frac{1}{j}.
$$
Let $I^t_n=I^{t, \frac{1}{j}}_n$ for $K_j<n\le K_{j+1}$. Then we have
$$
\lim_{n\to +\infty} \frac{\log\#(I^{t}_n)}{\log \beta_n}=0  \quad \textit{ and } \quad
\lim_{n\to\infty}\frac{\sum_{k\notin I^{t}_n}\overline c_{n, k}^t}{\sum_{k\in I^{t}_n}\overline c_{n, k}^t}=0.
$$
Let $\Psi_{t}$ be the subsystem generated by $I^{t}_n$.  For every $n\in\mathbb{N}$ and $d>t'>t>h$,  since  $\frac{\overline c_{n, k}}{\overline c_{n, \#I_n^t}}\ge 1$ for $k\in I^{t}_n$ and $\frac{\overline c_{n, k}}{\overline c_{n, \#I_n^t}}\le 1$ for $k\notin I^{t}_n$, it follows that
\begin{eqnarray*}
\frac{\sum_{k\notin I^{t}_n}\overline c_{n, k}^t}{\sum_{k\in I^{t}_n}\overline c_{n, k}^t}&=&    \frac{\sum_{k\notin I^{t}_n}(\frac{\overline c_{n, k}}{\overline c_{n, \#I_n^t}})^t}{\sum_{k\in I^{t}_n}(\frac{\overline c_{n, k}}{\overline c_{n, \#I_n^t}})^t}  \ge \frac{\sum_{k\notin I^{t}_n}(\frac{\overline c_{n, k}}{\overline c_{n, \#I_n^t}})^{t'}}{\sum_{k\in I^{t}_n}(\frac{\overline c_{n, k}}{\overline c_{n, \#I_n^t}})^{t'}} =\frac{\sum_{k\notin I^{t}_n}\overline c_{n, k}^{t'}}{\sum_{k\in I^{t}_n}\overline c_{n, k}^{t'}},
\end{eqnarray*}
which implies that for all  $t'>t$,
$$
\lim_{n\to\infty}\frac{\sum_{k\notin I^{t}_n}\overline c_{n, k}^{t'}}{\sum_{k\in I^{t}_n}\overline c_{n, k}^{t'}}=0.
$$     

By (1), we have that  either $\sum_{k\in I_n}\overline c_{n, k}^h<\infty$ for all $n$ or $\sum_{k\in I_n}\overline c_{n, k}^h=\infty$ for all $n$.          
If $\sum_{k\in I_n}\overline c_{n, k}^h<\infty$, by the same argument as above, we obtain a  subsystem $\Psi_h$ and take  $\Phi_h=\Psi_h$. If $\sum_{k\in I_n}\overline c_{n, k}^h=\infty$, fix a decreasing sequence $\{t_m\}_{m=1}^\infty$  convergent to $h$, and for each $m$, by the above argument, there exists an increasing sequence $\{L_m\}_{m=1}^\infty$ such that for each $n>L_m$,                   
$$
\frac{\log\#(I^{t_m}_n)}{\log \beta_n}<\frac{1}{m}  \qquad \textit{ and }\qquad
\frac{\sum_{k\notin I^{t_m}_n}\overline c_{n, k}^t}{\sum_{k\in I^{t_m}_n}\overline c_{n, k}^t}<\frac{1}{m}.
$$
 Let $I^h_n=I^{t_m}_n$ for $L_m<n\le L_{m+1}$,  and let $\Phi_{h}$ be the subsystem generated by $I^{h}_n$.

In both case, the subsystem $\Phi_h$ satisfies that for each $t$ with $\sum_{k\in I_n}\overline c_{n, k}^t<\infty$,
$$
\lim_{n\to +\infty} \frac{\log\#(I^{h}_n)}{\log \beta_n}=0  \qquad \textit{ and }\qquad
\lim_{n\to\infty}\frac{\sum_{k\notin I^{h}_n}\overline c_{n, k}^t}{\sum_{k\in I^{h}_n}\overline c_{n, k}^t}=0.
$$
By Lemma \ref{ap by finite}, we have $\underline P^h(t)=\underline P(t)$ for each $t$ with $\sum_{k\in I_n}\overline c_{n, k}^t<\infty$.

By a similar method, we may choose another sequence $J^h_n\subset I_n$,  and let $\Phi'_h$ be the corresponding subsystem, which  satisfies that for each $t$ with $\sum_{k\in I_n}\overline c_{n, k}^t=\infty$,
$$
\lim_{n\to +\infty} \frac{\log\#(J^{h}_n)}{\log \beta_n}=0 \qquad \textit{ and }\qquad
\lim_{n\to\infty}\sum_{k\in J^{h}_n}\overline c_{n, k}^t=\infty,
$$
which means $\underline P'^h(t)=\infty$ for $t$ with $\sum_{k\in I_n}\overline c_{n, k}^t=\infty$.
Let $I'_n=\max\{I^h_n, J^h_n\}$. Then the corresponding finite subsystem $\Phi'$ satisfies
$$
\lim_{n\to +\infty} \frac{\log\#(I'_n)}{\log \beta_n}=0,
$$
and for each $t\in (0, d)$,  we have
$\underline P'(t)=\underline P(t)$.

For $h=d$ or $h=0$,  the conclusion follows  by the same argument,.

Furthermore, if $\hdd E'=s_*'$, by $\underline P'(t)=\underline P(t)$, we have $s_*=s_*'$. Since $E'\subset E$, by Lemma \ref{ub hausdorff} it follows that
$$
\hdd E\le s_*=s_*'=\hdd E'\le \hdd E,
$$
that is  $\hdd E=\hdd E'=s_*.$ If $\pkd E'\ge {s^*}'$, by $\overline P'(t)=\overline P(t)$, we have $s^*={s^*} '$, which implies
$\ubd E\ge \pkd E\ge s^*.$
\end{proof}

\begin{cor}\label{limsup<1}
Let $E$ be the non-autonomous infinite set of $\Phi$ satisfying \eqref{cdn_cucl}.Given  a real sequence $\{\beta_n\}_{n=1}^\infty$ convergent to  zero. Suppose that   $\{\beta_n\}_{n=1}^\infty$ satisfies the following for all $t\in(0,d)$:
\begin{itemize}
\item[(1).]The sum $\sum_{k\in I_n}\overline c_{n, k}^t$ are either infinite for all $n$ or finite for all $n$.
\item[(2).] If $\sum_{k\in I_n}\overline c_{n, k}^t<\infty$, then
$$
\lim_{n\to\infty}\frac{\sum_{k\le {\beta_n^{-\frac{1}{2}}}}\overline c_{n, k}^t}{\sum_{j\in I_n}\overline c_{n, j}^t}=1.
$$
\item[(3).]If $\sum_{k\in I_n}\overline c_{n, k}^t=\infty$, then $\lim_{n\to\infty}\sum_{k\le {\beta_n^{-\frac{1}{2}}}}\overline c_{n, k}^t=\infty.$
\end{itemize}
Then there exists a   finite subsystem $\Phi'$ of $\Phi$ such that $\underline P'(t)=\underline P(t)$ for all $t\ge 0$ and
$$
\limsup_{k\to\infty}\frac{\log \#I_k'}{\log \beta_n}<1.
$$
Furthermore, let   $E'$ be the non-autonomous set of $\Phi'$, and  let  $s_* '$ be given by \eqref{s^*} with respect to $\Phi'$. If $\hdd E'=s_*'$, then $\hdd E=s_*'=s_*.$
\end{cor}

For a non-autonomous finite iterated function system, the following lemma shows that we may  approximate it with subsystems where the contraction ratios  are bounded by the number of mappings at each level.  It is easier   to estimate the dimension of a non-autonomous infinite iterated function system in Section \ref{sec_InfIFS}. The following  conclusion is inspired by   \cite{RM}.
\begin{lem}\label{subsystem}
Given $t_0>0$ and  a non-autonomous finite iterated function system $\Phi$ satisfying \eqref{cdn_cucl}. Let $\{\alpha_n\}_{n=1}^\infty $ be a real sequence strictly increasing  to $\infty$. Then there exists a subsystem $\Phi'$ of $\Phi$ such that $\overline P'(t)=\overline P(t), \underline P'(t)=\underline P(t)$ for all $t\ge t_0$ and for all integers $n>0$
$$
\max_{k,l\in I'_n}\frac{\overline c_{n, k}}{\overline c_{n, l}}\le\alpha_n(\#I_n)^\frac{1}{t_0}.
$$
\end{lem}

\begin{proof}
Let $I'_n$ be a set  consisting  of all indices $j\in I_n$ such that
$$
\overline c_{n, j}\ge \max_j\{\overline c_{n, j}\}\alpha_n^{-1}\#I_n^{-\frac{1}{t_0}}.
$$
Then for $t\ge t_0$,
$$
\sum_{u_j\in I_j- I_j'}\overline c_{u_j}^t\le (\max_j\{\overline c_{n, j}\}\alpha_n^{-1}\#I_n^{-\frac{1}{t_0}})^t\#I_n\le\alpha_n^{-t}\sum_{u_j\in I_j'}\overline c_{u_j}^t.
$$
Given $\delta >0$, since $\lim_{n\to\infty} \alpha_n =\infty$,  for sufficiently large $n$, we have that
$$
\sum_{u_j\in  I_j'}\overline c_{u_j}^t\le (1+\delta) \sum_{u_j\in I_j}\overline c_{u_j}^t.
$$

Since  $\#I_n<\infty$ for each $n\in\mathbb{N}_+$,  it follows that $\sum_{k\in I_n}\overline c_{n, k}^t<\infty$ for all $n$.
By  Lemma \ref{ap by finite}, we have $\overline P'(t)=\overline P(t)$ and $\underline P'(t)=\underline P(t)$ for all $t\ge t_0$.
\end{proof}

\section{Pressure functions of non-autonomous  conformal iterated function systems}\label{sec_PF}

In this section, we define a pressure function of non-autonomous conformal iterated function systems, which is a powerful tool to study fractal dimensions.

\begin{lem}\label{u contraction}
Given a $NCIFS$, for each integer $k$ and $u\in I_k$, 
$$
\|D\varphi_u\|\le c,
$$
where $c$ is given by \eqref{def_uccdn}.
\end{lem}

\begin{proof}
Since $\overline{\mbox{int}(J)}=J$, for every $x\in J$ there exists $\{x_n\}_{n=1}^\infty\subset J$ such that $\lim_{n\to\infty}x_n=x$. Note that each $\varphi_u$ is a $C^1$ conformal diffeomorphism of $V$ into $V$. Thus, we have
$$
\|D\varphi_u(x)\| = \lim_{y \to x} \frac{|\varphi_u(y) - \varphi_u(x)|}{|y - x|}=\lim_{n\to\infty} \frac{|\varphi_u(x_n) - \varphi_u(x)|}{|x_n - x|}.
$$
By (i) in definition, it follows that
$$
|\varphi_{u}(x)-\varphi_{u}(y)|\le  c|x-y| \qquad \textit{ for all }  x, y\in J,
$$
and $\|D\varphi_u\|\le c$.
\end{proof}

\begin{lem}\label{ddx1}
Given a non-autonomous finite conformal iterated function system $\Phi$.  For every $\theta\in(0,1]$, both upper and lower pressure functions are strictly monotonous in $t$, that is for $t_1 <t_2$ if $\overline P(t_1, \theta)$ and $\overline P(t_2, \theta)$ are finite, we have  $\overline P(t_1, \theta) > \overline P(t_2, \theta)$, where $\overline P(t, \theta)$ is given in \eqref{def_PTF}. 
\end{lem}

\begin{proof}
Given $\theta\in(0,1]$ and $t_1<t_2$ such that $\overline P(t_1, \theta)$ and $\overline P(t_2, \theta)$ are finite. By \eqref{def_uccdn}, we have
\begin{eqnarray*}
\min_{\mathcal{M}\in\mathcal{M}(\delta, \theta)}\{\sum_{\mathbf{u}\in\mathcal{M}}\|D\varphi_\mathbf{u}\|^{t_2}\}
&\le&c^{k_{\delta}(t_2-t_1)}\min_{\mathcal{M}\in\mathcal{M}(\delta, \theta)}\{\sum_{\mathbf{u}\in\mathcal{M}}\|D\varphi_\mathbf{u}\|^{t_1}\},
\end{eqnarray*}
where $c$ is given by \eqref{def_uccdn}, and it follows that
$$
\overline P(t_2, \theta)\le\overline P(t_1, \theta)-(t_2-t_1)\log\frac{1}{c}.
$$
Then we have $\overline P(t_2, \theta) < \overline P(t_1, \theta)$, since $0<c<1$.
\end{proof}

\begin{cor}\label{ddx}
Let $\Phi$ be a non-autonomous finite or infinite conformal iterated function system. Both upper and lower pressure functions are strictly monotonous, that is for $t_1 <t_2$ if $\overline P(t_1)$ and $\overline P(t_2)$ are finite, we have  $\overline P(t_1) > \overline P(t_2)$, where $\overline P(t)$ is given in \eqref{def_PFT2}.
\end{cor}

The following conclusions are the consequences of  bounded distortion (see (II') in Subsection \ref{subsec_NCIFS}),  which are frequently used in our proofs.

\begin{lem}\label{cor_subMul}
Let $\Phi$ be a non-autonomous finite or infinite conformal iterated function system. For all integers $0< m< n<\infty$, every  $\bu\in\Sigma^m$ and $\bu\in\Sigma_{m+1}^n$,  we have that
\begin{equation}
C^{-1}\|D\varphi_{\bu}\|\|D\varphi_{\bv}\|\le \|D\varphi_{\bu\bv}\|\le C \|D\varphi_{\bu}\|\|D\varphi_{\bv}\|.
\end{equation}
\end{lem}	

\begin{proof}
Since $D\varphi_{\mathbf{u}\bv}(x)=D\varphi_{\mathbf{u}}(\varphi_{\bv}(x))D\varphi_{\bv}(x)$, it follows that
$$
\|D\varphi_{\bu\bv}\|\le \|D\varphi_{\bu}\|\|D\varphi_{\bv}\|.
$$

By Bounded distortion, for every $x\in J$, we have that 
\begin{eqnarray*}
\|D\varphi_\mathbf{uv}\|&\ge& \|D\varphi_{\mathbf{u}}(\varphi_{\bv}(x))D\varphi_{\bv}(x)\|  \\
&=&\|D\varphi_{\mathbf{u}}(\varphi_{\bv}(x))\|\|D\varphi_{\bv}(x)\|  \\
&\ge&C^{-1}\|D\varphi_{\mathbf{u}}\|\|D\varphi_{\bv}(x)\|,
\end{eqnarray*}
and the conclusion holds.
\end{proof}

\begin{lem}[Quasi-differential Mean Value Theorem]\label{quasi_diff_mvt}
Let $\Omega \subseteq \mathbb{R}^d$ be an open convex set, and let $f: \Omega \to \mathbb{R}^d$ be a differentiable mapping. Then for all distinct points $x, y \in \Omega$, there exist a point $\xi$ on the line segment connecting $x$ and $y$ such that
$$
|f(x) - f(y)|\le \|Df(\xi)\||x - y|.
$$
\end{lem}
The principle of bounded distortion makes precise the idea of a set being 
'approximately nonautonomous similar', in that any sufficiently small neighbourhood may 
be mapped onto a large part of the set by a transformation that is not unduly 
distorting. 

\begin{lem}\label{contraction}
Given NCIFS $\boldsymbol{\Phi}$, for all $\mathbf{u}\in\Sigma^*$,  we have that for all $x, y\in J$,
\begin{equation}\label{ineq_BDxyL}
\|D\varphi_\mathbf{u}\||x-y|\asymp|\varphi_\mathbf{u}(x)-\varphi_\mathbf{u}(y)|, 
\end{equation}
and moreover,
\begin{equation}
\|D\varphi_\mathbf{u}\|\asymp|J_\mathbf{u}|. 
\end{equation}
\end{lem}	

\begin{proof}
Recall that $J$ is compact and $V$ is an open connected set containing $J$. The collection $\mathcal{F}=\{B(x, r_x)\subset V:x\in J\}$   of  balls is a cover of  $J$. Thus, there exists  a finite subcover $\{B(x_i, r_i)\}_{i=1}^k\subset\mathcal{F}$ of $J$. Let $\delta$ be the Lebesgue number of $\{B(x_i, r_i)\}_{i=1}^k$;  see \cite{PW}. Since $V$ is a connected set of $  \mathbb{R}^d$, it is also path-connected. For every $1\le i\le k-1$, there exists a path connecting $x_i$ and $x_{i+1}$, and we choose a finite number of balls $\{B(z_j, r_j)\}_{j=1}^{K_i}$ with  $z_1=x_i$ and $z_{K_i}=x_{i+1}$ satisfying  $B(z_j, r_j)\cap B(z_{j+1}, r_{j+1})\not=\emptyset$. We denote the new collection of these balls by $\{B(z_j, r_j)\}_{j=1}^l$.  Note that for each $j $, $B(z_j, r_j)\cap B(z_{j+1}, r_{j+1})\not=\emptyset$.  

Fix $\bu\in\Sigma^*$. Arbitrarily choose  $x, y\in J$. There exist integers $m$ and  $n$ such that $x\in B(z_m, r_m), y\in B(z_n, r_n)$.

We first  show  
\begin{equation}\label{ineq_BDxyR}
|\varphi_\mathbf{u}(x)-\varphi_\mathbf{u}(y)|\le C' \|D\varphi_\mathbf{u}\||x-y|.
\end{equation}
If $0<|x-y|<\delta$,  then by Lemma \ref{quasi_diff_mvt}, it holds.  
Otherwise, for  $|x-y|\geq \delta$,  by Lemma \ref{quasi_diff_mvt}, we have
\begin{eqnarray*}
|\varphi_\mathbf{u}(x)-\varphi_\mathbf{u}(y)|&\le&|\varphi_\mathbf{u}(x)-\varphi_\mathbf{u}(z_m)|+\cdots+|\varphi_\mathbf{u}(z_n)-\varphi_\mathbf{u}(y)|   \\
&\le&\|D\varphi_\mathbf{u}\||x-z_m|+\cdots+\|D\varphi_\mathbf{u}\|z_n-y|  \\
&\le& 2\delta(\|D\varphi_\mathbf{u}\|+\cdots+2\|D\varphi_\mathbf{u}\|) \\
&\le&C'\|D\varphi_\mathbf{u}\||x-y|.
\end{eqnarray*}
Hence  \eqref{ineq_BDxyR} holds.

Next, we show  \begin{equation}\label{ineq_BDxyL}
|\varphi_\mathbf{u}(x)-\varphi_\mathbf{u}(y)|\ge C'' \|D\varphi_\mathbf{u}\||x-y|.
\end{equation} 
Suppose $\varphi_\bu(x)\in B(z_p, r_p)$ and $\varphi_\bu(y)\in B(z_q, r_q)$ where $1\le p\le q\le l$. 
If $p = q$, let $L$ be the line segment connecting $\varphi_\bu(x)$ and $\varphi_\bu(y)$. If $p < q$, let $L$ be the polyline connecting the points in the following order:
$\varphi_\bu(x)$, $z_p$, $z_{p+1}$, $\dots$, $z_q$, $\varphi_\bu(y)$. We may parameterize $L$ by a continuous map $L:[0,1]\to\mathbb{R}^d$ with $L(0)=\varphi_\bu(x)$ 
and $L(1)=\varphi_\bu(y)$. For $s\in[0,1]$, define $L_s = \{L(t):0\le t\le s\}$ and let 
$|L_s|$ denote the length of $L_s$.  Note that $|L|\leq 2\sum_{j=1}^lr_j$

Let $t_0=\sup\{t\in[0,1] :L_t\subset\varphi_\bu(V)\}$. Then
$$
|L|\ge   |L_{t_0}|\ge C^{-1} \|D\varphi_\bu\|\it{dist}(\partial V, J)\frac{|x-y|}{|J|}.     
$$
Since $\varphi_\bu(x), \varphi_\bu(y)\in J$, if $|\varphi_\bu(x)-\varphi_\bu(y)|\le\delta$, then $|\varphi_\bu(x)-\varphi_\bu(y)|=|L|$, and if $|\varphi_\bu(x)-\varphi_\bu(y)|>\delta$, then   
$$
\frac{2\sum_{j=1}^lr_j}{\delta}|\varphi_\bu(x)-\varphi_\bu(y)|\ge2\sum_{j=1}^lr_j\ge|L|.
$$              
Thus
$$
|\varphi_\bu(x)-\varphi_\bu(y)|\ge\frac{\delta \it{dist}(\partial V, J)}{2C\sum_{j=1}^lr_j|J|}\|D\varphi_\mathbf{u}\||x-y|.
$$

Since $\Psi_\bu(x)=\varphi_\bu(x)+\omega_\bu$, it follows that
$$
|\Psi_\mathbf{u}(x)-\Psi_\mathbf{u}(y)|=|\varphi_\mathbf{u}(x)-\varphi_\mathbf{u}(y)|,
$$
and we have $\|D\varphi_\mathbf{u}\||x-y|\asymp|\varphi_\mathbf{u}(x)-\varphi_\mathbf{u}(y)|$.
\end{proof}

\begin{lem}\label{vol}
Let $\Phi$ be a non-autonomous finite or infinite conformal iterated function system. For all $\mathbf{u}\in\Sigma^*$ and $A\subset V$ we have
\begin{equation}
C^{-d}\|D\varphi_\mathbf{u}\|^d\mathcal{L}^d(A)\le\mathcal{L}^d(\Psi_\mathbf{u}(A))\le C^d\|D\varphi_\mathbf{u}\|^d\mathcal{L}^d(A).
\end{equation}
\end{lem}	

\begin{proof}
Since $\varphi_\bu$ is a $C^1$ conformal diffeomorphism, it is clear that
$$
\mathcal{L}^d(\Psi_\mathbf{u}(A))=\int_{A}\|D\varphi_\mathbf{u}(x)\|^dd\mathcal{L}^d.
$$
By bounded distortion, we have
$$
C^{-d}\|D\varphi_\mathbf{u}\|^d\mathcal{L}^d(A)\le\mathcal{L}^d(\Psi_\mathbf{u}(A))\le \|D\varphi_\mathbf{u}\|^d\mathcal{L}^d(A),
$$
and the conclusion holds.
\end{proof}

\section{ intermediate dimension  of non-autonomous finite conformal set} \label{sec_ID}
In this section, we study intermediate dimensions of the non-autonomous finite conformal set $E$ generated by the NFCIFS $\Phi$. Recall
$$
\mathcal{M}(\delta, \theta)=\{\mathcal{M}:\mathcal{M} \textit{ is a cut set satisfying }\delta^\frac{1}{\theta} <  \|D\varphi_{\mathbf{u}^*}\| \textit{ and }   \|D\varphi_\mathbf{u}\|\le \delta \textit { for each } \mathbf{u} \in \mathcal{M}\},
$$

For  $F\subset \mathbb{R}^d$ such that  $E\cap F\ne \emptyset$,   we write
\begin{equation}\label{def_Aset}
A(F)=\{\mathbf{u} \in \Sigma^* : J_\mathbf{u}\cap F \ne \emptyset,  \|D\varphi_\mathbf{u}\|\le\delta< \|D\varphi_{\mathbf{u}^*}\| \}.
\end{equation}
Let
\begin{equation}\label{def_k0}
k_F^- = \min\{k:|\mathbf{u}|=k, \mathbf{u}\in A(F)\}, \qquad
k_F^+ = \max\{k:|\mathbf{u}|=k, \mathbf{u}\in A(F)\}.
\end{equation}
For each integer $k_F^-\le k\leq k_F^+$, we write
\begin{equation}\label{def_DFk}
D(F, k)=\{\mathbf{u} \in \Sigma^k :\mathbf{u} \in A(F)\}.
\end{equation}

\begin{lem}\label{finite intersection}
Let $\Phi$ be a non-autonomous finite conformal iterated function system satisfying OSC.  Then there exists  a constant $C_2$ such that for  every  $F \subset \mathbb{R}^d$ with  $E\cap F\ne \emptyset$, we have
$$
\sum_{k = k_F^-}^{k_F^+}\underline c_k^d\# D(F, k)\le C_2.
$$
\end{lem}

\begin{proof}
Given a set $F \subset \mathbb{R}^d$ such that   $E\cap F\ne \emptyset$. Let $\delta=|F|$. For every $\bu\in \mathcal{M}(\delta) $, we write $\bu=u_1u_2\ldots u_k$.
By  Lemma \ref{contraction}, we  have $ \delta \leq  \|D\varphi_{\mathbf{u}^*}\|. $
Recall $\underline c_k=\min_{1\le j \le \# I_k} \{\|D\varphi_{k, j}\|\}$, and by Lemma \ref{cor_subMul}, it is clear that for all $x\in J$
$$
\|D\varphi_\mathbf{u}\| \ge \|D\varphi_{\mathbf{u}^*}(\varphi_{u_k}(x))D\varphi_{u_k}(x)\|  \ge  C^{-1}\|D\varphi_{\mathbf{u}^*}\|\underline c_k \ge C_1^{-1}C^{-1}\underline c_k\delta.
$$
Arbitrarily choose $x\in F$,  and  we have $J_\mathbf{u} \subset B(x,2\delta)$ for every $ J_\mathbf{u} \in A(F)$. By Lemma \ref{contraction} and Lemma \ref{vol}, it immediately follows that
\begin{eqnarray*}
\frac{ \mathcal{L}^d(\mbox{int}(J))}{(C_1C^2)^{d}}\sum_{k = k_F^-}^{k_F^+}\underline c_k^d\# D(F, k)\delta ^d &\le&\frac{ \mathcal{L}^d(\mbox{int}(J))}{C^{d}}\sum_{k = k_F^-}^{k_F^+}\sum_{\mathbf{u} \in D(F, k)} \|D\varphi_\mathbf{u}\|^d   \\
&=&\frac{ \mathcal{L}^d(\mbox{int}(J))}{C^{d}} \sum_{\mathbf{u} \in A(F)} \|D\varphi_\mathbf{u}\|^d  \\
&\le& \mathcal{L}^d(B(x,2\delta).
\end{eqnarray*}
Setting    $C_2=\frac{(2C_1 C^2)^d \mathcal{L}^d(B(0,1))}{ \mathcal{L}^d((\mbox{int}J))}$,  we obtain $\sum_{k = k_F^-}^{k_F^+}\underline c_k^d\# D(F, k)\le C_2 $.
\end{proof}

For each cut set $ \mathcal{M}$ satisfying $ \delta^{\frac{1}{\theta}} <\|D\varphi_{\mathbf{u}^*}\| $  and $ \|D\varphi_\mathbf{u}\|\le \delta$ for all $\mathbf{u} \in \mathcal{M}$. 
Let $\mathcal{M}'=\{\bu\in\mathcal{M}:|J_\bu|\le\delta\}$ and $\mathcal{M}''=\{\bu\in\mathcal{M}:|J_\bu|>\delta\}$.

We define $\mathcal{F}_{\mathcal{M}'}=\{U_{\mathbf{u}}: \mathbf{u}\in \mathcal{M}'\}$ by setting
		$$
		U_\mathbf{u}=\left\{ \begin{array}{ll}
		J_\mathbf{u}& \textit{ if } \delta^{\frac{1}{\theta}}\le |J_\mathbf{u}|\le\delta, \\
		\bigcup_{x\in J_\mathbf{u}} B(x,\frac{\delta^\frac{1}{\theta}- |J_{\mathbf{u}}|}{2}) & \textit{ if } |J_\mathbf{u}|<\delta^{\frac{1}{\theta}}.
		\end{array} \right.
		$$
For $\bu\in\mathcal{M}''$, we define a cover $\mathcal{F}_\mathcal{\bu}=\{U_{\bu_i}\}_{i=1}^{k_\bu}$ of $J_\bu$ satisfying $|U_{\bu_i}|=\delta$. 

Let 
$$
\mathcal{F}_\mathcal{M}=\mathcal{F}_\mathcal{M'}\cup\bigcup_{\bu\in\mathcal{M}''}\mathcal{F}_\bu.
$$

\begin{proof}[Proof of Theorem \ref{thm_IMdim}]
We only give the proof for the lower intermediate dimension since the proof for the upper intermediate dimension is similar.
	
First, we prove $\lid E\geq s_\theta$. Arbitrarily choose $\alpha<s_\theta$. Recall that $s_\theta =\sup\{t:\underline P(t, \theta)>0\}$, and we have $\underline P(\alpha, \theta)>0$. There exists $\Delta_1 >0$ such that   for all  $0<\delta <\Delta_1$, we have that
\begin{equation}\label{eq 1}
 \min_{\mathcal{M}\in\mathcal{M}(\delta, \theta)}\{\sum_{\mathbf{u}\in\mathcal{M}}\|D\varphi_\mathbf{u}\|^\alpha\}>e^\frac{k_{\delta}\underline P(\alpha, \theta)}{2}>1.
\end{equation}
  Since $\lim_{k\to +\infty} \frac{\log\underline c_k}{\log M_k}=0$, for each $0<\eta<1$, there exists $K_0>0$ such that  for $k>K_0$, we have
\begin{equation}\label{Mk<ck1}
M_k^\eta<\underline c_k^d.
\end{equation}
Moreover there exists $\Delta_2$ such that for all $0<\delta<\Delta_2$,  we have $|\mathbf{u}|>K_0$  for all $\bu\in\Sigma^*$ satisfying $\|D\varphi_\mathbf{u}\|\le\delta$. 

Given a  cover  $\mathcal{F}$ of $E$ such that  $ {\delta^\frac{1}{\theta}}\le | U| \le \delta$ for each $U\in \mathcal{F}$.  By Lemma \ref{finite intersection},  there exists  a constant $C_2$ such that for  every  $U\in \mathcal{F} $ with  $E\cap U\ne \emptyset$, we have
\begin{equation}\label{sum<C1}
		\sum_{k = k_U^-}^{k_U^+}\underline c_k^d\# D(U, k)\le C_2,
\end{equation}
where $k_U^-$ is given by \eqref{def_k0}. By \eqref{def_Aset},  for $\bu\in A(U)$, it follows that $$
\sum_{\mathbf{u} \in D(U, k)} \|D\varphi_\mathbf{u}\|^\alpha \leq \sum_{\mathbf{u} \in D(U, k)}  M_k^\eta \lvert U\rvert^{\alpha-\eta} \leq \underline c_k^d\# D(U, k) |U\rvert^{\alpha-\eta}
$$
  Combining this with \eqref{Mk<ck1} and \eqref{sum<C1},  we have for $k_U^- >K_0$
		\begin{eqnarray*}
	\sum_{ U\in\mathcal{F}} \sum_{\mathbf{u}\in A(U)} \|D\varphi_\mathbf{u}\|^\alpha  &=&\sum_{U\in\mathcal{F}} \sum_{k = k_U^-}^{k_U^+}\sum_{\mathbf{u} \in D(U, k)} \|D\varphi_\mathbf{u}\|^\alpha  \le C'\sum_{U\in\mathcal{F}} \lvert U\rvert^{\alpha-\eta},
		\end{eqnarray*}
where $C'$ is a constant independent of $\mathcal{F}$.	
	
Let $	\mathcal{F_1} =\{\mathbf{u}:\mathbf{u}\in A(U), U\in\mathcal{F}\}$. It is clear that $\mathcal{F_1}$ is a cut set satisfying $\delta^\frac{1}{\theta} < \|D\varphi_{\mathbf{u}^*}\|$ and $\|D\varphi_{\mathbf{u}}\|\leq \delta$. Moreover, we choose a finite cut set
$$
\{{\mathbf{u}_i}\}_{i = 1}^{n}\subset \bigcup_{U\in\mathcal{F}} A(U)
$$
such that $\{\mathbf{u}_i\}_{i=1}^n\subset\mathcal{M}(\delta, \theta).$
By \eqref{eq 1}, it follows that
$$
\sum_{ U\in\mathcal{F}} \sum_{\mathbf{u}\in A(U)} \|D\varphi_\mathbf{u}\|^\alpha  \geq \sum_{i = 1}^{n} \|D\varphi_{\mathbf{u}_i}\|^\alpha >\min_{\mathcal{M}\in\mathcal{M}(\delta, \theta)}\{\sum_{\mathbf{u}\in\mathcal{M}}\|D\varphi_\mathbf{u}\|^\alpha\}>1,
$$		
and it implies that
$$
\sum_{U\in\mathcal{F}} \lvert U\rvert^{\alpha-\eta}\geq \frac{1}{C'}.
$$

Setting $\epsilon_0=\frac{1}{C'}$ and $\Delta_0=\min\{\Delta_1, \Delta_2\}$,  for every $0<\delta<\Delta_0$  and every cover $\mathcal{F}$ satisfying  $ {\delta^\frac{1}{\theta}}\le | U| \le \delta$ for all $U\in \mathcal{F}$,    we have that
		$$
		\sum_{U\in\mathcal{F}} \lvert U\rvert^{\alpha-\eta} \geq \epsilon_0.
		$$
		It implies that $ \lid E \ge \alpha-\eta$. Since $\eta >0$ and $\alpha < s^\theta$ are  arbitrarily chosen, we obtain that $  \lid E \ge s_\theta.$

		Next, we prove $\lid E\leq s_\theta$. Arbitrarily choose $\beta>\gamma>s_\theta$.
Since $\gamma>s_\theta$, that is $\underline P(\gamma, \theta)<0$, there exists a sequence $\{\delta_k\}_{k=1}^\infty$ convergent to $0$ such that
$$
 \min_{\mathcal{M}\in\mathcal{M}(\delta_k, \theta)}\{\sum_{\mathbf{u}\in\mathcal{M}}\|D\varphi_\mathbf{u}\|^\gamma\}<e^\frac{p_{\delta_k, \theta}\underline P(\gamma)}{2}<1.
$$
For each $k>0$, there exists a  cut set $\mathcal{M}_k$ such that $\delta_k^{\frac{1}{\theta}} < \|D\varphi_{\mathbf{u}^*}\|$   and $ \|D\varphi_\mathbf{u}\|\le \delta_k$ for all $\mathbf{u} \in \mathcal{M}_k$  satisfying
\begin{equation}\label{skd=1}
\sum_{\mathbf{u}\in \mathcal{M}_k} \|D\varphi_\mathbf{u}\|^\gamma <1 .
\end{equation}
By  \eqref{condition}, there exists an integer  $K_1>0$, such that  for all  $k>K_1$,
\begin{equation}\label{MC<1}
\underline c_k^{- \beta}M_k^{\beta-\gamma} <1.
\end{equation}
By  Lemma \ref{cor_subMul} and Lemma \ref{contraction},  we have that
$$
\|D\varphi_{\bu^*}\| \leq  C (\underline{c}_{|\bu|})^{-1}\|D\varphi_\mathbf{u}\|
$$
Since $|J_\bu|\le C_1\delta$ there exists a constant $C''$ such that For each integer  $k>0$ and each $\bu\in\mathcal{M}_k''$, 
$$
\#\mathcal{M}_k''\le C''.
$$
 Since $\beta>\gamma>s_\theta$, by \eqref{skd=1} and \eqref{MC<1},  we obtain that
\begin{eqnarray*}
	\sum_{U\in \mathcal{F}_{\mathcal{M}_k}} | U|^{\beta}&\leq &C''\sum_{\mathbf{u}\in \mathcal{M}_k}  \underline c_{|\bu|}^{-\beta} \|D\varphi_\mathbf{u}\|^{\beta} \leq C''\sum_{\mathbf{u}\in \mathcal{M}_k} \|D\varphi_\mathbf{u}\|^{\gamma}    \Big(\underline c_{|\bu|}^{-\beta} M_{|\mathbf{u}|}^{\beta-\gamma} \Big) <C''.
		\end{eqnarray*}
		It follows that $\lid E \le \beta$.  Since $\beta \geq s^\theta$ is arbitrarily chosen, we obtain that	$\lid E \le s^\theta$.
\end{proof}

\section{box dimension and packing dimension of non-autonomous finite conformal set}\label{sec_PAB}

Let $E$ be the non-autonomous finite conformal set of the NFCIFS $\Phi$. According to the definition of the intermediate dimension, the box dimension is obtained by $s^\theta$ at $\theta=1$, and we may show the relationship between the box dimension and the jump point of the upper pressure function.

\begin{lem}\label{box d}
Let $E$ be the non-autonomous finite conformal set of the NFCIFS $\Phi$ satisfying  \eqref{condition} and OSC. Then  the box dimensions of $E$  is given by
$$
\ubd E=s^*,
$$
where $s^*$ is given by \eqref{def_s*}.
\end{lem}

\begin{proof}
For $\theta=1$, by theorem \ref{thm_IMdim}, we have $\ubd E=s^1$. Recall
$$
t^*=\inf\{t:\sum_{k=1}^\infty\sum_{\mathbf{u}\in\Sigma^k}\|D\varphi_\mathbf{u}\|^t<\infty\}.
$$
Then we may show $s^1\le t^*\le s^*\le s^1$. First we show $s^*\ge t^*$. For each non-integral $t>s^*$, we have
$\overline P(t)<0$. Hence there exists a $K_0>0$ such that for each $k>K_0$, 
$$
\sum_{\mathbf{u}\in\Sigma^k}\|D\varphi_\mathbf{u}\|^t<e^{k\frac{\overline P(t)}{2}},
$$
and it follows that
$$
\sum_{k=K_0}^\infty\sum_{\mathbf{u}\in\Sigma^k}\|D\varphi_\mathbf{u}\|^t<\sum_{k=K_0}^\infty e^{k\frac{\overline P(t)}{2}}<\infty.
$$
Then we have $t>t^*$, and $s^*\ge t^*$.

For every $t>t^*$, we have $\sum_{k=1}^\infty\sum_{\mathbf{u}\in\Sigma^k}\|D\varphi_\mathbf{u}\|^t<\infty,$
and it follows that
$$
\sum_{\mathbf{u}\in\mathcal{M}}\|D\varphi_\mathbf{u}\|^t\le\sum_{k=1}^\infty\sum_{\mathbf{u}\in\Sigma^k}\|D\varphi_\mathbf{u}\|^t<\infty,
$$
which means $\overline P(t, 1)\le 0$. Then we have $t>s^1$, and $t^*\ge s^1$.

For each $t'>t>s^1$, we have $\overline P(t, 1)<0$, and there exists a $\Delta>0$ such that for each $0<\delta<\Delta$,
$$
\sum_{\mathbf{u}\in\mathcal{M}(\delta)}\|D\varphi_\mathbf{u}\|^t<e^{k_\delta\frac{\overline P(t, 1)}{2}}<1.
$$                    
Choosing $\rho$ such that $c<\rho<1$ where $c$ is the uniform contraction constant given by \eqref{def_uccdn}, there exists an integer $k_2$ such that  $\rho^{k_2+1}<\Delta\le\rho^{k_2}$.
For each integer  $k>0$,  we write
$$
\mathcal{M}(\rho^k)=\{\bu\in \Sigma^*:    \|D\varphi_\mathbf{u}\|\le\rho^k <  \|D\varphi_{\mathbf{u}^*}\|\},
$$
Since  $c<\rho<1$, it is clear that   $\Sigma^*=\bigcup_{k=0}^\infty\mathcal{M}(\rho^k)$,  and   it follows that
\begin{eqnarray*}
\sum_{\mathbf{u}\in\Sigma^k}\|D\varphi_\mathbf{u}\|^{t'} &\le&\sum_{k=1}^\infty \sum_{\mathbf{u}\in\mathcal{M}(\rho^k)}\|D\varphi_\mathbf{u}\|^{t} \|D\varphi_\mathbf{u}\|^{t'-t}  \\
&\le&\sum_{k=1}^{k_2} \sum_{\mathbf{u}\in\mathcal{M}(\rho^k)}\|D\varphi_\mathbf{u}\|^t +\sum_{k=k_2}^\infty (C_1\rho^k)^{t'-t}  \\
&<&\infty,
\end{eqnarray*}
Hence, by \eqref{def_PFT2}, it is clear that $\overline P(t')\le 0$, and by \eqref{def_s*}, we have that  $t'>s^*$ and $s^1\ge s^*$. Therefore $s^*=s^1=\ubd E.$
\end{proof}

Recall that $J_\bu=\Psi_\bu(J)$ for each $\mathbf{u}\in\Sigma^*$.  We write
$$
E^\bu=\bigcap_{k=1}^\infty\bigcup_{\bv\in\Sigma_{|\bu|+1}^{|\bu|+k}}J_{\bu\bv},
$$
and it is clear that $E^\bu\subset E$. Note that $E^\bu$ may be regarded as the attractor of a certain non-autonomous iterated function system $\Phi^\bu$, and we denote $\overline P^\bu(t)$ as its corresponding upper topological pressures. In the following conclusion, we actually  show  that
$$
\ubd E^\bu= \ubd E.
$$
\begin{proof}[Proof of Theorem \ref{thm_PuBdim}]
For each open set $O\subset \mathbb{R}^d$ such that  $O\cap E\not=\emptyset$,  we choose $x\in O\cap E$, and by Lemma \ref{Sigma=E}, there exists $\bu\in\Sigma^\infty$ such that $x=\pi_\Phi(\mathbf{u}).$ Since $\lim_{k\to\infty}|J_{\mathbf{u}|_k}|=0$, there exists an integer  $k>0$ such that $J_{\mathbf{u}|_k}\subset O$, and   we write $\bu_k=\bu|_k\in\Sigma^k$. 
For each $\bv\in\Sigma^k$, we have
$$
\overline P^\bv(t)=\limsup_{n\to \infty} \frac{1}{n-k}\log \sum_{\mathbf{u}\in \Sigma^n_{k+1}}\|D\varphi_\mathbf{vu}\|^t.
$$
Since  $C^{-1}\|D\varphi_\mathbf{v}\|\|D\varphi_\mathbf{u}\|\le\|D\varphi_\mathbf{vu}\|\le \|D\varphi_\mathbf{v}\|\|D\varphi_\mathbf{u}\|,$
we have
$$
\overline P^\bv(t)=\limsup_{n\to \infty} \frac{1}{n}\log \sum_{\mathbf{u}\in \Sigma^n_{k+1}}\|D\varphi_\mathbf{u}\|^t.
$$
Let
$$
s_\bv^*=\inf\{t: \overline P^\bv(t)<0\}=\sup\{t: \overline P^\bv(t)>0\}.
$$
By Lemma \ref{box d}, it follows that $\ubd E^\mathbf{w}=s^*_\mathbf{w}=s^*_\bv=\ubd E^\mathbf{v}$ for each $\bv, \mathbf{w}\in\Sigma^k$, and by finite stability of box dimensions, we have
$$
\ubd E^\mathbf{v}=\ubd E^\mathbf{w}=\ubd E,
$$
for all $\mathbf{v}, \mathbf{w}\in\Sigma^k$.
Thus we have $ \ubd E^{\bu_k}=\ubd E, $ and this implies that $\ubd (E\cap O)=\ubd E $  for every open set $O\cap E\not=\emptyset$. By \cite[Corollary 3.10]{Fal}, it follows  that $\pkd E=\ubd E=s^*$.
\end{proof}

By Theorem \ref{thm_PuBdim},  we may relax the condition to \eqref{def_cmc} for the lower bound of the upper box dimensions.

\begin{prop}\label{lb packing}
Let  $E$ be the non-autonomous finite conformal set satisfying OSC and
$$
\lim_{k\to +\infty} \frac{\log\overline c_k-\log\#I_k}{\log M_k}=0.
$$
Then the box and packing dimension of $E$ bounded below $s^*$, that is
$$
\ubd E\ge\pkd E\ge s^*,
$$
\end{prop}

\begin{proof}
Fix $t_0<s_*$, and let $\alpha_k=k$. By Lemma \ref{subsystem}, there exists a subsystem $\Phi'$ of $\Phi$ satisfying $\overline c_k=\overline c_k'$, $\overline P'(t)=\overline P(t)$ for all $t\ge t_0$ and for all $n$
$$
\max_{k,l\in I'_n}\frac{\|D\varphi_{n, k}\|}{\|D\varphi_{n, l}\|}\le\alpha_n(\#I_n)^\frac{1}{t_0}.
$$
Let $E'$ be the non-autonomous finite conformal set of $\Phi'$.

Since
$$
\frac{\overline c'_k}{\underline c'_k}=\max_{m,l\in I'_k}\frac{\|D\varphi_{k, m}\|}{\|D\varphi_{k, l}\|}\le\alpha_k(\#I_k)^\frac{1}{t_0},
$$
it follows that
$$
\frac{\log\underline c'_k}{\log M_k}\le\frac{\log\overline c'_k-\log\alpha_k-\log(\#I_k)^\frac{1}{t_0}}{\log M_k}.
$$
Since  $\|D\varphi_\bu\|\le c^k$ for each $\bu\in\Sigma^k$ where $c$ is uniform contraction constant given by \eqref{def_uccdn}, we have
$$
0\le\lim_{k\to\infty}\frac{-\log\alpha_k}{\log M_k}\le\lim_{k\to\infty}\frac{-\log\alpha_k}{k\log c}=0.
$$
Note that  $M_k'\leq M_k<1$, and it follows that
\begin{eqnarray*}
0&\le&\lim_{k\to\infty}\frac{\log\underline c'_k}{\log M'_k}  \le\lim_{k\to\infty}\frac{\log\overline c_k-\log\alpha_k-\log(\#I_k)^\frac{1}{t_0}}{\log M_k}
=0.
\end{eqnarray*}
By Theorem \ref{thm_PuBdim}, $\pkd E'={s^*}'$. Note that $\underline P'(t)=\underline P(t)$ for all $t\ge t_0$, that is ${s^*}'=s^*$. Then we have $\ubd E\ge\pkd E\ge\pkd E'= s^*$.
\end{proof}

\section{Hausdorff dimension of non-autonomous finite conformal set}\label{sec_HF}
In this section, we first study the Hausdorff dimension of non-autonomous finite conformal sets under  $\mathcal{L}^d(\partial J)=0$ and \eqref{condition}.
Then we  further relax the latter condition to
$$
\lim_{k\to +\infty} \frac{\log\overline c_k-\log\#I_k}{\log M_k}=0.
$$
where $M_k, \underline c_k$ and $\overline c_k$ are given by \eqref{Mk}.

Lemma \ref{ub hausdorff} shows that $s_*$ given by \eqref{s_*}  is always an  upper bound to the Hausdorff dimensions of non-autonomous finite conformal sets.
Next, we show that $s_*$ actually gives the Hausdorff dimensions of non-autonomous finite conformal sets under the condition  $\mathcal{L}^d(\partial J)=0$ and
$ \lim_{k\to +\infty} \frac{\log\underline c_k}{\log M_k}=0.$
\begin{proof}[Proof of Theorem \ref{thm_Hdim}]
By Lemma \ref{ub hausdorff}, it is sufficient to prove $\hdd E\ge s_*$.

For each  $t'< t <s_*$, we have $\underline P(t)>0$, and  there exists $K_1$ such that
\begin{equation}\label{dengbi}
\sum_{\mathbf{u}\in\Sigma^k}\|D\varphi_\mathbf{u}\|^t>e^{\frac{\underline P(t)}{2}k},
\end{equation}
 for all $k>K_1$. Since $\mbox{int}(J_\bu)\cap  \mbox{int}(J_\bv)=\emptyset$ and $\mathcal{L}^d(\partial J)=0$, we have that $\mathcal{L}^d(J_\bu\cap J_\bv)=0$ for all $\bu, \bv\in\Sigma^*$ with $|\bu|=|\bv|$.
Since $\lim_{k\to +\infty} \frac{\log\underline c_k}{\log M_k}=0, $ there exists an integer  $K_0>0$ such that  for all  $k>K_0$,
\begin{equation}\label{1MC<1}
M_k^{t-t'} <\underline{c}_k^{d}.
\end{equation}

Next,  we  construct a  sequence of probability measures $\{\mu_n\}_{n=1}^\infty$ on $J$. For each $n>0$, Let $\mu_n$ be a uniformly distributed measure given by
\begin{equation}\label{def_msu}
\mu_n(J_\mathbf{u})=\frac{\|D\varphi_\mathbf{u}\|^t}{\sum_{\mathbf{u}\in\Sigma^n}\|D\varphi_\mathbf{u}\|^t},
\end{equation}
for all  $J_\bu, \bu\in\Sigma^n$. For every $0< m< n$  and every  $\bu\in\Sigma^m$, we have that
$$
\mu_n(J_\mathbf{u}) =    \sum_{\mathbf{v}\in\Sigma^{n}_{m+1}}\mu_n(J_{\mathbf{u}\mathbf{v}}) = \sum_{\mathbf{v}\in\Sigma^{n}_{m+1}}\frac{\|D\varphi_{\mathbf{u}\bv}\|^t }{\sum_{\mathbf{u}\in\Sigma^n}\|D\varphi_\mathbf{u}\|^t}.
$$
By Lemma \ref{cor_subMul}, it follows that
\begin{eqnarray*}
\mu_n(J_\mathbf{u}) &\le&\frac{\|D\varphi_{\mathbf{u}}\|^t\sum_{\mathbf{v}\in\Sigma^{n}_{m+1}}\|D\varphi_\mathbf{v}\|^t }{C^{-2t}\sum_{\mathbf{u}\in\Sigma^{m}}\|D\varphi_\mathbf{u}\|^t\sum_{\mathbf{v}\in\Sigma^{n}_{m+1}}\|D\varphi_\mathbf{v}\|^t}  =\frac{C^{2t}\|D\varphi_{\mathbf{u}}\|^t}{\sum_{\mathbf{u}\in\Sigma^m}\|D\varphi_\mathbf{u}\|^t}.
\end{eqnarray*}
For every  $B(x, r)\subset\mathbb{R}^d$, by \eqref{def_Aset}, we have
$$
\mu_n(B(x,r))\le\mu_n(\bigcup_{\mathbf{u}\in A(B(x,r))}J_\mathbf{u})  =   \sum_{\mathbf{u}\in A(B(x,r))} \mu_n(J_\mathbf{u}).
$$
By \eqref{def_k0} and \eqref{def_DFk}, letting $n>\max\{K_1, K_0, k^+_{B(x,r)}\}$, it implies  that
\begin{equation}\label{mu equality}
\mu_n(B(x,r)\le\sum_{k=k_{B(x,r)}^-}^{k_{B(x,r)}^+}\sum_{\mathbf{u}\in D(B(x,r), k)}\frac{\|D\varphi_\mathbf{u}\|^t C^{2t}}{\sum_{\bu\in\Sigma^k}\|D\varphi_\bu\|^t,}.
\end{equation}
Combining  \eqref{dengbi}, \eqref{1MC<1} with Lemma \ref{contraction}, it follows that
\begin{eqnarray*}
\sum_{\mathbf{u}\in D(B(x,r), k)}\frac{\|D\varphi_\mathbf{u}\|^t C^{2t}}{\sum_{\bu\in\Sigma^k}\|D\varphi_\bu\|^t,}&\le&\sum_{\mathbf{u}\in D(B(x,r), k)}\frac{(C_1|J_\mathbf{u}|)^{t'}M_k^{t-t'} C^{2t}}{\sum_{\bu\in\Sigma^k}\|D\varphi_\bu\|^t,} \\
&\le&e^{\frac{-\underline P(t)}{2}}(2rC_1)^{t'}C^{2t}\#D(B(x,r),k)\underline c_k^{d}  \\
&=&C_3r^{t'}\#D(B(x,r),k)\underline c_k^{d},
\end{eqnarray*}
where $C_3=(2C_1)^{t'}C^{2t}e^{\frac{-\underline P(t)}{2}}$ is a constant.
By Lemma \ref{finite intersection}, we have
\begin{equation}\label{mu<rt}
\mu_n(B(x,r))\le\sum_{k=k_{B(x,r)}^-}^{k_{B(x,r)}^+}C_3r^{t'}\#D(B(x,r),k)\underline c_k^{d} \le
C_2C_3r^{t'}.
\end{equation}

Since $\{\mu_n\}$ is bounded on $J$, there exists a subsequence $\{\mu_{n_k}\}$ converging weakly to a measure $\mu$; see \cite{BP}. Hence,   taking limit on both sides of \eqref{mu<rt}, we obtain that
$$
\mu(B(x,r)\le\liminf_{n\to\infty}\mu_n(B(x,r))\le C_2C_3 r^{t'}.
$$
By the  mass distribution principle  \cite{Fal}, this implies  that $\hdd E\ge t'$ for all $t'<s_*$, and we have $\hdd E\ge s_*$, which completes the proof.
\end{proof}

Next, we show that $s_*$ still gives the Hausdoff dimension under \eqref{def_cmc} by applying Theorem \ref{thm_Hdim} and approximation of subsystems.

\begin{proof}[Proof of Corollary \ref{thm_Hdim g}]
By Lemma \ref{ub hausdorff}, if $s_*=0$, the conclusion holds.
If $s_*>0$, by the similar argument to Proposition \ref{lb packing}, there exists a subsystem $\Phi'$ of $\Phi$ such that $\underline P'(t)=\underline P(t)$ for all $t\ge t_0$, and  $E'$ is the non-autonomous finite conformal set of $\Phi'$ satisfying  $\lim_{k\to\infty}\frac{\log\underline c'_k}{\log M'_k}=0.$
By Theorem \ref{thm_Hdim}, $\hdd E'=s_*'$. Note that $\underline P'(t)=\underline P(t)$ for all $t\ge t_0$, that is $s_*'=s_*$. Since $\hdd E'\le\hdd E\le s_*$, we have $\hdd E=s_*$.
\end{proof}

Given a set $J\subset \mathbb{R}^d$, we say $J$ satisfies \emph{cone condition} if there exists $\alpha,l>0$ such that for every $x\in\partial J\subset \mathbb{R}^d$, there exists an open cone $\mbox{Con}(x, u_x, \alpha, l)\subset\mbox{int}(J)$ with vertex $x$, direction vector $u_x$, central angle of Lebesgue measure $\alpha$ and altitude $l$. Given a non-autonomous finite or infinite conformal iterated function system $\Phi$, we say $\Phi$ satisfies \emph{cone condition} if the initial set $J$ in the definition of $\Phi$ satisfies \emph{cone condition}.  See \cite{DM} for details. It is a standard fact that  the cone condition implies   $\mathcal{L}^d(\partial J)=0$, and we include a proof for the convenience of readers.

\begin{lem}\label{cone}
Given $\Phi$ satisfying  cone condition and OSC. For every $x\in J$, $r>0$, Let $\{\bu_\lambda\}_{\lambda\in\Lambda}\subset \Sigma^*$ satisfy that $|J_{\bu_\lambda}|>r$ and $
J_{\bu_\lambda}\cap B(x,r)\not=\emptyset$ for $\lambda\in\Lambda$, and that $\Psi_{\bu_{\lambda_1}}(\mbox{int}(J))\cap \Psi_{\bu_{\lambda_2}}(\mbox{int}(J))=\emptyset $ for $ \lambda_1\not=\lambda_2$. Then 
$$
\# \Lambda\leq M_1,
$$
where $M_1$ is a constant independent of $x$ and $r$. 
\end{lem}
\begin{proof}
By \cite[(2.10)]{DM}, there exists a constant $D>1$ and $0<\beta\le\alpha$ such that for all $x\in J$ and for all $\bu\in\Sigma^*$, by \eqref{basic set},   it follows that
$$
\mbox{Con}(\Psi_\bu(x), \beta, D^{-1}\|D\varphi_\bu\|)\subset\Psi_\bu(\mbox{int}(J)).
$$
Since $|J_{\bu_\lambda}|>r$, it is clear that $\{\bu_\lambda\}_{\lambda\in \Lambda}$ is finite, i.e. $\# \Lambda <\infty$. For each $\lambda\in\Lambda$, since $B(x, r)\cap\Psi_{\bu_\lambda}(J)\not=\emptyset$, by the  bounded distortion,   there exists $\mbox{Con}_\lambda\subset  \Psi_{\bu_\lambda}(\mbox{int}(J))$ with  central angle of Lebesgue measure $\beta$ and altitude $(C_1D)^{-1}r$ such that  $\mbox{Con}_\lambda\subset B(x,(1+(C_1D)^{-1}r)$, and  $\mbox{Con}_\lambda\cap \mbox{Con}_{\lambda'}=\emptyset$  for all $\lambda\neq \lambda'$.  It follows that
$$
\# \Lambda\beta ((C_1D)^{-1}r)^d =\# \Lambda\mathcal{L}^d(\mbox{Con}_\lambda)   = \sum_{\lambda\in\Lambda} \mathcal{L}^d(\mbox{Con}_\lambda)  \le \mathcal{L}^d(B(x,(r+(C_1D)^{-1}r)).
$$
Let $M_1=\frac{(1+(C_1D)^{-1})^d}{\beta ((C_1D)^{-1})^d}$,  and the conclusion holds.  
\end{proof}

\begin{lem}\label{c L=0}
Let $J\subset\mathbb{R}^d$ be a compact set satisfying cone condition. Then we have $\mathcal{L}^d(\partial J)=0$.
\end{lem}

\begin{proof}
Let $Q=\{x_1, \ldots, x_{d-1}, 0)\in\mathbb{R}^d:x_i\in\mathbb{Q}, 1\le i\le d-1\}$. For each  $x\in\partial J\subset \mathbb{R}^d$, by cone condition, there exists an open cone $\mbox{Con}(x, u_x, \alpha, l)\subset\mbox{int}(J)$ with vertex $x$, direction vector $u_x$, central angle of Lebesgue measure $\alpha$ and altitude $l$. Since  $\mathbb{Q}$  is   dense in $\mathbb{R}$,   there exists another open cone $\mbox{Con}(x, u_x', \frac{\alpha}{2}, \frac{l}{2})\subset\mbox{int}(J)$ such that the  line passing through the point $x$ with the direction of $u_x'$ intersects $Q$.             For each $x\in\partial J$, let $l_x$ be the straight line passing through the point $x$ with the direction of $u_x'$.  

Let $k=[\frac{2|J|}{l}]+1$. Fix $q\in Q$. For each  $1\le i\le k$, let
\[
\begin{split}
A_{q, +}^i &=\{x\in\partial J:u_x'=\frac{x-q}{|x-q|}, (i-1)\frac{l}{2}\le|x-q|<i\frac{l}{2}\}, \\
A_{q, -}^i &=\{x\in\partial J:u_x'=\frac{q-x}{|x-q|}, (i-1)\frac{l}{2}\le|x-q|<i\frac{l}{2}\}.
\end{split}
\]
For $x\in A_{q, +}^i$, let
\[
\begin{split}
B_{q, +}^i(x) &=\mbox{Con}(x, u_x', \frac{\alpha}{2}, \frac{l}{2})\bigcup\mbox{Con}(q, -u_x', |x-q|-\frac{l}{2}), \\
B_{q, -}^i(x) &=\mbox{Con}(x, u_x', \frac{\alpha}{2}, \frac{l}{2})\bigcup\mbox{Con}(q, u_x', |x-q|+\frac{l}{2}),
\end{split}
\]
where  $\mbox{Con}(q, u_x', |x-q|+\frac{l}{2})$  and $\mbox{Con}(q, -u_x', |x-q|-\frac{l}{2})$ have the same base as $\mbox{Con}(x, u_x', \frac{\alpha}{2}, \frac{l}{2})$.

Since $Q$ is countable,  we  rewrite it as $Q=\{q_j\}_{j=1}^\infty$. For each $j\in\mathbb{N}$, let $B_{q_j, +}^i=\bigcup_{x\in A_{q_j, +}^i}B_{q_j, +}^i(x).$ For $0<\epsilon<\frac{l}{2|J|}$, we have
$$
A_{q_j, +}^i-q_j\subset(\mbox{int}((1+\frac{\epsilon}{2^j})(B_{q_j, +}^i-q_j))-\mbox{int}(B_{q_j, +}^i-q_j)),
$$
and it follows that
\begin{equation} \label{msAq1}
\mathcal{L}^d(A_{q_j, +}^i)\le((1+\frac{\epsilon}{2^j})^d-1)\mathcal{L}^d(\mbox{int}B_{q_j, +}^i)\le((1+\frac{\epsilon}{2^j})^d-1)\mathcal{L}^d(B(0,|J|+1)).
\end{equation}
Similarly,  we have that
\begin{equation}\label{msAq2}
\mathcal{L}^d(A_{q_j, -}^i)\le(1-(1-\frac{\epsilon}{2^j})^d)\mathcal{L}^d(\mbox{int}B_{q_j, -}^i)\le(1-(1-\frac{\epsilon}{2^j})^d)\mathcal{L}^d(B(0,|J|+1)).
\end{equation}
Since
$$
\partial J\subset\bigcup_{j=1}^\infty \bigcup_{i=1}^k(A_{q_j, +}^i\cup A_{q_j, -}^i),
$$
it  implies that
$$
\mathcal{L}^d(\partial J)\le\sum_{j=1}^\infty\sum_{i=1}^k(\mathcal{L}^d((A_{q_j, +}^i)+\mathcal{L}^d((A_{q_j, -}^i))  \le\sum_{j=1}^\infty k(\mathcal{L}^d((A_{q_j, +}^i)+\mathcal{L}^d((A_{q_j, -}^i)),
$$
and combining it with \eqref{msAq1} and \eqref{msAq2}, we obtain that
\begin{eqnarray*}
\mathcal{L}^d(\partial J) &\le&k\mathcal{L}^d(B(0,|J|+1))\sum_{j=1}^\infty4d\frac{\epsilon}{2^j}  \leq C_4\epsilon,
\end{eqnarray*}
where $C_4$ is  a constant. The conclusion follows by the arbitrariness of $\epsilon$.
\end{proof}

In \cite{RM}, Rempe-Gillen and Urbański showed if  a non-autonomous finite conformal set satisfies
 cone condition  and $\lim_{k\to\infty}\frac{\log \#I_k}{k}=0$, then its Hausdorff dimension  is equal to $s_*$.
We slightly generalise their conclusion into the following form.
\begin{cor}\label{mphi}
Given NFCIFS $\Phi$ satisfying cone condition and OSC. Suppose  $m_\Phi=\inf\{\underline P(t): \underline P(t)>0\}>0$.
If
\begin{equation} \label{cdn_Imk}
\limsup_{k\to +\infty} \frac{\log\#I_k}{k \cdot m_\Phi }<1,
\end{equation}
then $\hdd E=s_*,$ where $s_*$ is given by \eqref{def_s*}.
\end{cor}

\begin{proof}
By Lemma \ref{ub hausdorff}, it is sufficient to prove $\hdd E\ge s_*$. For each  $t<s_*$, we have $\underline P(t)>0$.
By  \eqref{cdn_Imk},   there exist two constants  $M\in(0, 1)$ and  $K>0$ such that for all $k>K$,
\begin{equation}\label{1MC<11}
\#I_k<e^{Mm_\Phi  \cdot k}  \quad \textit{and } \quad
\sum_{\mathbf{u}\in\Sigma^k}\|D\varphi_\mathbf{u}\|^t>e^{M\underline P(t)k}.
\end{equation}
Since $\Phi$ satisfies the cone condition, by Lemma \ref{c L=0}, we construct the  measure $\mu_n$   by \eqref{def_msu}. Given  $x\in J$ and $r>0$,    by \eqref{def_DFk},  for each $\bu\in D(B(x,r), k)$,  it is clear  that $|J_{\bu^*}|>2r$ and $J_{\bu^*}\cap B(x,r)\not= \emptyset$.
Since $\Phi$ satisfies cone condition, by Lemma \ref{cone}, we have
$$
\#D(B(x,r), k)\le M_1 \#I_k.
$$

Recall \eqref{mu equality}, that is
$$
\mu_n(B(x,r)\le\sum_{k=k_{B(x,r)}^-}^{k_{B(x,r)}^+}\sum_{\mathbf{u}\in D(B(x,r), k)}\frac{\|D\varphi_\mathbf{u}\|^t C^{2t}}{\sum_{\bu\in\Sigma^k}\|D\varphi_\bu\|^t}.
$$
By Lemma \ref{contraction}, we have $\|D\varphi_\bu\|\leq C_1|J_\bu|\le C_1 r$. Combining it with  \eqref{1MC<11}, we have that
$$
\sum_{\mathbf{u}\in D(B(x,r), k)}\frac{\|D\varphi_\mathbf{u}\|^t C^{2t}}{\sum_{\bu\in\Sigma^k}\|D\varphi_\bu\|^t}\leq  \frac{(C_1r)^{t} C^{2t}}{e^{M\underline P(t)k}}  M_1 \#I_k\leq C_5 r^{t} e^{M(m_\Phi-\underline P(t))k},
$$
where $C_5$ is a constant independent of $r$  and $k$.
Choose $n>\max\{K, k^+_{B(x,r)}\}$. By \eqref{def_Aset},  it follows that
\begin{eqnarray*}
\mu_n(B(x,r)) &\le& C_5 r^{t} \sum_{k=k_{B(x,r)}^-}^{k_{B(x,r)}^+}{e^{M(m_\Phi-\underline P(t))k}} \le C_5\sum_{k=0}^{\infty}{e^{M(m_\Phi-\underline P(t))k}}r^t.
\end{eqnarray*}

Since $\{\mu_n\}$ is bounded on $J$, there exists $\{\mu_{n_k}\}$ weakly convergent  to a measure $\mu$ (see \cite{BP}), and it follows that
$$
\mu(B(x,r)\le\liminf_{n\to\infty}\mu_n(B(x,r))\le C_5\sum_{k=0}^{\infty}{e^{M(m_\Phi-\underline P(t))k}} r^{t}.
$$
By the  mass distribution principle in \cite{Fal}, it follows that $\hdd E\ge t'$ for all $t<s_*$, and we have $\hdd E\ge s_*$.
\end{proof}

\section{Dimensions of the non-autonomous infinite conformal set} \label{sec_InfIFS}

In this section, we study the Hausdorff dimensions of non-autonomous infinite conformal sets by applying the strategy in Subsection \ref{subsec IIFS}.

Let $\Phi$ be a non-autonomous infinite conformal iterated function system. Recall that  $M_n=\max_{\mathbf{u} \in \Sigma^n} \{\|D\varphi_\mathbf{u}\|\}$ and
\begin{eqnarray*}
s^*&=&\inf\{t: \overline P(t)<0\}=\sup\{t: \overline P(t)>0\},  \\
s_* &=&\inf\{t: \underline P(t)<0\}=\sup\{t: \underline P(t)>0\}.
\end{eqnarray*}

We  assume that $\{\|D\varphi_{n,k}\|\}_{k=1}^\infty$ is monotonically decreasing, and we have
$$
\overline c_n=\|D\varphi_{n, 1}\|.
$$

\begin{proof}[Proof of Theorem \ref{thm_infinite}]
By Proposition \ref{lim=0}, there exists a finite subsystem $\Phi'$ satisfying $\lim_{n\to +\infty} \frac{\log\#(I'_n)}{\log M'_n}=0$
and $\underline P'(t)=\underline P(t)$ for each $t\in (0, d)$.
Let $E'$ be the non-autonomous finite conformal set of $\Phi'$. Since $\overline c_n= \overline c_n'$ and $M_n\ge M_n'$, we have $\lim_{n\to\infty}\frac{\log\underline c'_n}{\log M'_n} =0.$
By Proposition \ref{lb packing}, we have
$$
\ubd E\ge\pkd E\ge\pkd E'\ge s^*.
$$

Furthermore, since  $\mathcal{L}^d(\partial J)=0$,
by Lemma \ref{ub hausdorff} and Corollary \ref{thm_Hdim g}, we have $\hdd E=\hdd E'=s_*$.
\end{proof}

Given  the non-autonomous infinite conformal system $\Phi$ satisfying the cone condition.   Rempe-Gillen and Urbański  \cite{RM}  showed that if   the following hold for all $t\in(0,d)$ and all $\epsilon>0$:
\begin{itemize}
\item[(1).]The sum $\sum_{k\in I_n}\|D\varphi_{n, k}\|^t$ are either infinite for all $n$ or finite for all $n$.
\item[(2).]If $\sum_{k\in I_n}\|D\varphi_{n, k}\|^t<\infty$, then $
\lim_{n\to\infty}\frac{\sum_{k\le e^{\epsilon n}}\|D\varphi_{n, k}\|^t}{\sum_{k\in I_n}\|D\varphi_{n, k}\|^t}=1.
$
\item[(3).]If $\sum_{k\in I_n}\|D\varphi_{n, k}\|^t=\infty$, then $\lim_{n\to\infty}\sum_{k\le e^{\epsilon n}}\|D\varphi_{n, k}\|^t=\infty.$
\end{itemize}
Then the Hausdorff dimension of the non-autonomous conformal set is equal to $s_*$.

Compared with their conclusions, since $\lim_{k\to\infty}\frac{k}{-\log M_k}\le\frac{1}{-\log c}$, we   reduce the requirements for the convergence rate of $\sum_{k\in I_n}\|D\varphi_{n, k}\|^t$, but we need the condition $\lim_{k\to +\infty} \frac{\log\overline c_k}{\log M_k}=0$
holds in Theorem \ref{thm_infinite}.

The following conclusion is a consequence of  Corollary \ref{mphi}. 

\begin{cor}\label{Cor_RUI}
Let $E$ be the non-autonomous infinite conformal set of the NICIFS $\Phi$ satisfying OSC and cone condition with $m_\Phi=\inf\{\underline P(t): \underline P(t)>0\}>0$.
Suppose there exists $0<M<1$ such that for all $t\in(0,d)$, it satisfies
\begin{itemize}
\item[(1).]The sum $\sum_{k\in I_n}\|D\varphi_{n, k}\|^t$ are either infinite for all $n$ or finite for all $n$.
\item[(2).] If $\sum_{k\in I_n}\|D\varphi_{n, k}\|^t<\infty$, then
$$
\lim_{n\to\infty}\frac{\sum_{k\le {e^{Mm_\Phi n}}}\|D\varphi_{n, k}\|^t}{\sum_{j\in I_n}\|D\varphi_{n, j}\|^t}=1.
$$
\item[(3).]If $\sum_{j\in I_n}\|D\varphi_{n, k}\|^t=\infty$, then $\lim_{n\to\infty}\sum_{k\le {e^{Mm_\Phi n}}}\|D\varphi_{n, k}\|^t=\infty.$
\end{itemize}
Then we have $\hdd E=s_*$, where $s_*$ is given by \eqref{def_s*}.
\end{cor}

\begin{proof}
Since $E$ satisfies OSC, by Theorem  \ref{ub hausdorff}, it  follows that $\hdd E \leq s_*$.

Let $\beta_n =e^{-2M m_\Phi n}$. By Corollary \ref{limsup<1}, there exists a finite subsystem $\Phi'$ such that
$$
\limsup_{n\to +\infty} \frac{\log\#(I'_n)}{m_\Phi n}<1.
$$
 Since $\Phi'$ satisfies OSC and cone condition, by   Corollary \ref{mphi}, we have $\hdd E'=s_*'$, where $s_*'$ is given by \eqref{def_s*} with respect to the subsystem $\Phi'$. By Corollary \ref{limsup<1}, we obtain $\hdd E \geq \hdd E'=s_*'=s_*$, and the conclusion holds.
\end{proof}

\end{document}